\documentclass[11pt]{amsart}

\usepackage[foot]{amsaddr}

\calclayout

\usepackage[utf8]{inputenc} % allow utf-8 input
\usepackage[T1]{fontenc}    % use 8-bit T1 fonts
\usepackage{hyperref}       % hyperlinks
\usepackage{url}            % simple URL typesetting
\usepackage{booktabs}       % professional-quality tables
\usepackage{amsfonts}       % blackboard math symbols
\usepackage{nicefrac}       % compact symbols for 1/2, etc.
\usepackage{microtype}      % microtypography
\usepackage{lipsum}

\usepackage{amsmath,amssymb,amsthm}
\usepackage{graphicx}
\usepackage[colorinlistoftodos]{todonotes}
\usepackage{thmtools}
\usepackage{thm-restate}
\usepackage{wrapfig}
\newtheorem{theorem}{Theorem}

\newtheorem{lemma}[theorem]{Lemma}

\def\ie{\textit{i.e.}}
\newcommand{\norm}[1]{\left\Vert#1\right\Vert}

\newcommand{\abs}[1]{\left\vert#1\right\vert}

\newcommand{\set}[1]{\left\{#1\right\}}
\newcommand{\parr}[1]{\left (#1\right )}
\newcommand{\brac}[1]{\left [#1\right ]}
\newcommand{\ip}[1]{\left \langle #1 \right \rangle }

\newcommand{\Real}{\mathbb R}
\newcommand{\Natural}{\mathbb N}
\newcommand{\Lap}{L}
\newcommand{\NCone}{B}

\def \conv{{\mathrm{Conv}}}
\def \x{{\mathrm{x}}}
\def \y{{\mathrm{y}}}

\def \NN{{\mathcal{N}}}
\def \TT{{\mathcal{T}}}
\def \PP{{\mathcal{P}}}
\def \CC{{\mathcal{C}}}

\newcommand{\DISK}{\mathbb B}
\newcommand{\IMG}{\Omega}
\newcommand{\MAP}{\phi}
\newcommand{\BMAP}{\gamma}
\newcommand{\BND}{\Gamma}
\newcommand{\CUSPS}{S}

\title{Non-Convex Planar Harmonic Maps}

\author{Shahar Z. Kovalsky$^1$}
\address{$^1$Duke University}

\author{Noam Aigerman$^2$}
\address{$^2$Adobe Research}

\author{Ingrid Daubechies$^1$}

\author{Michael Kazhdan$^3$}
\address{$^3$Johns Hopkins University}

\author{Jianfeng Lu$^1$}

\author{Stefan Steinerberger$^4$}
\address{$^4$Yale University}

\begin{document}
\maketitle

\begin{abstract}
We formulate a novel characterization of a family of invertible maps between two-dimensional domains. Our work follows two classic results: The Rad\'o-Kneser-Choquet (RKC) theorem, which establishes the invertibility of harmonic maps into a convex planer domain; and Tutte's embedding theorem for planar graphs - RKC's discrete counterpart - which proves the invertibility of piecewise linear maps of triangulated domains satisfying a  discrete-harmonic principle, into a convex planar polygon. In both theorems, the convexity of the target domain is essential for ensuring invertibility. We extend these characterizations, in both the continuous and discrete cases, by replacing convexity with a less restrictive condition. In the continuous case, Alessandrini and Nesi provide a characterization of invertible harmonic maps into non-convex domains with a smooth boundary by adding additional conditions on orientation preservation along the boundary. We extend their results by defining a condition on the normal derivatives along the boundary, which we call the cone condition; this condition is tractable and geometrically intuitive, encoding a weak notion of local invertibility. The cone condition enables us to extend Alessandrini and Nesi to the case of harmonic maps into non-convex domains with a piecewise-smooth boundary. In the discrete case, we use an analog of the cone condition to characterize invertible discrete-harmonic piecewise-linear maps of triangulations. This gives an analog of our continuous results and characterizes invertible discrete-harmonic maps in terms of the orientation of triangles incident on the boundary.
\end{abstract}

% keywords can be removed

%%%%%%%%%%%%%%%%%%%%%%%%%%%%%%%%%%%%%%%%%%%%%%%%%%%%%%%%%%%%%%%%%%%%%%%%%%%%%%%%%%%%%%%
\section{Introduction}
This paper studies invertible maps, \ie, maps for which an inverse mapping exists. These maps are also sometimes called 1-to-1 maps, since they define a 1-to-1 correspondence between domains. This correspondence is crucial in many fields, for instance in enabling the study of variations of properties across a collection of objects: in the context of computational anatomy, considering different specimens of an organ  and studying the variation of corresponding regions aids the classification of the specimen as healthy or unhealthy. Similarly, in the context of evolutionary biology, variation across a collection is instrumental in the construction of phylogenetic trees describing how species evolve.

\begin{figure}[t!]
  \centering
  \includegraphics[width=1\columnwidth]{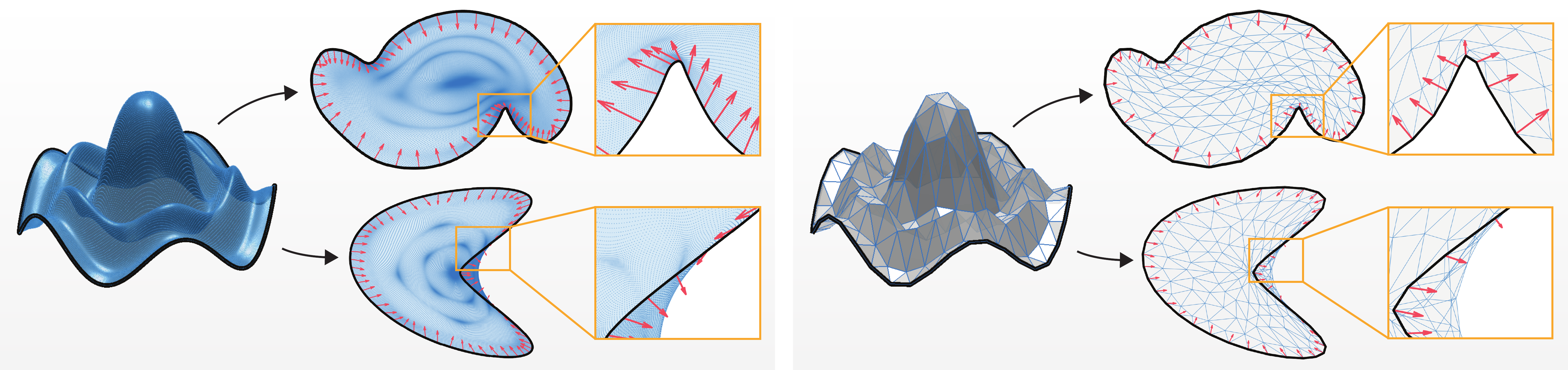}
  \caption{Planar harmonic maps into non-convex domains. Left panel shows a surface harmonically mapped into two different non-convex domains. Red arrows depict the normal gradients along the boundary; these can be thought of as the ``forces'' exerted by a system of springs or a membrane along the fixed boundary. In the top row, the forces along the boundary all point inwards, thus satisfying the cone condition we propose in this work which guarantees the resulting harmonic embedding is invertible. The second row shows a non-invertible harmonic mapping into a non-convex domain; in this case, ``forces'' point outwards as the spring/membrane is pulled out of the target domain, causing self-overlaps. The right panel shows that the exact same result holds for a \emph{discrete} harmonic embedding of a triangulated surface.}
  \label{fig:teaser}
\end{figure} 
 
Mathematicians have long tried to find families of maps that are guaranteed to be invertible, \ie, prove that if some criteria hold for a map, it is invertible. This is also the focus of our work. Namely, this paper concerns harmonic maps between two-dimensional domains (little is known regarding invertibility of harmonic maps into high-dimensional domains, and existing counterexamples indicate that the higher-dimensional case requires much more intricate assumptions). We consider two analogous families of maps for two separate cases - the continuous case, in which one continuous domain is mapped into another continuous domain; and the discrete case, in which a triangulated surface is mapped into a polygonal domain. 

One of the largest known families of invertible harmonic maps is characterized by the celebrated Rad\'o-Kneser-Choquet (RKC) theorem \cite{duren2004harmonic, rado1926aufgave, kneser1926losung, choquet1945type}. RKC considers \emph{harmonic maps} of planar disks into other planar domains (\ie, maps whose coordinate functions are solution of Laplace's equation). RKC states that given a homeomorphism (\ie, a continuous map with a continuous inverse) of the unit circle onto a simple closed curve bounding a \emph{convex} region, extending it harmonically to the interior produces a homeomorphism of the unit disk onto the interior of the curve. Though the constraint that the source domain is a disk may appear overly restrictive, taken in conjunction with the Riemann Mapping Theorem, RKC implies that a harmonic map of any compact Riemann surface with disk topology into a convex domain, such that the map between the one-dimensional boundaries is a homeomorphism, is itself a homeomorphism of the two-dimensional domains. 

In the discrete case, a closely-related characterization of invertible maps is derived from Tutte's work on planar embeddings of graphs \cite{tutte1963draw}. Tutte considered planar 3-connected graphs and their barycentric embeddings, in which the position of each vertex is set to the average of its neighbors' positions. When the outer face is a \emph{convex} polygon, Tutte proved that the edges of such a barycentric embedding, realized as straight lines, do not cross. 

Maps of triangulations have been further studied in \cite{floater1997parametrization, floater2003one,gortler2006discrete}, which consider discrete-harmonic embeddings of triangulated surfaces with disk topology. As in Tutte's embedding, maps are specified by prescribing the position of vertices within the target domain, and are extended to the interior of triangles through linear interpolation. These works show that if the image of the boundary of a triangulation is a convex polygon, then the interpolated map is a homeomorphism. This property of discrete-harmonic maps, along with their computational tractability (via solving a single sparse system of linear equations), accounts for their broad applicability in computer graphics, including meshing, shape analysis, parameterization, and texture mapping \cite{floater2005surface,liu2008local,weber2014locally,aigerman2014lifted,wong2015computing,bright2017harmonic,jiang2017simplicial,li2019optcuts,shen2019progressive}.

\subsection*{Approach}
In this paper, we extend the characterization of invertible (planar) harmonic maps. To this end, we consider harmonic and discrete-harmonic maps into \emph{non-convex} regions, and study conditions under which such maps are homeomorphisms. This enables the definition of correspondences while providing more degrees of freedom for optimizing the quality of the map (e.g. the extent to which the map distorts the notion of lengths, angles, areas).

One cannot simply omit the convexity requirement: in both the continuous (RKC) and discrete (Tutte) theorems, convexity plays a crucial role. For example, in the continuous case, it is known (Choquet \cite{choquet1945type}) that for any non-convex region bounded by a simple curve, there exists some homeomorphism from the unit circle to the boundary curve for which the harmonic extension to the unit disk fails to be a homeomorphism \cite{choquet1945type,alessandrini2009invertible}. Figure~\ref{fig:teaser} illustrates harmonic and discrete-harmonic maps into non-convex domains.

Hence, our work is driven by the following question: is there a property sufficient to guarantee injectivity, which is trivially satisfied in the case of convex boundaries, but can also be satisfied in non-convex cases? As we will show, the answer is affirmative and the conditions are tractable.

To give an informal intuition for our condition, imagine a system of springs with a fixed boundary; once the system reaches equilibrium, interior forces cancel out (this is the physical interpretation of the harmonicity condition), except along the boundary where the forces acting on the boundary are balanced by external forces holding the boundary in place (the Dirichlet boundary conditions). At any point along the boundary, if the interior forces point inwards, the membrane is pulled inwards and locally should be in the interior of the target domain. If, on the other hand, the internal forces point outward, the membrane is pulled out of the target domain, and folds over the boundary to spill out. See Figure~\ref{fig:teaser} for an illustration.

We formalize this intuition in terms of normal derivatives and prove that the condition requiring forces to point inwards (which is trivially satisfied for convex target domains) is sufficient to prove invertibility, independently of the target boundary's shape.

\begin{figure}[t!]
  \centering
  \includegraphics[width=0.8\columnwidth]{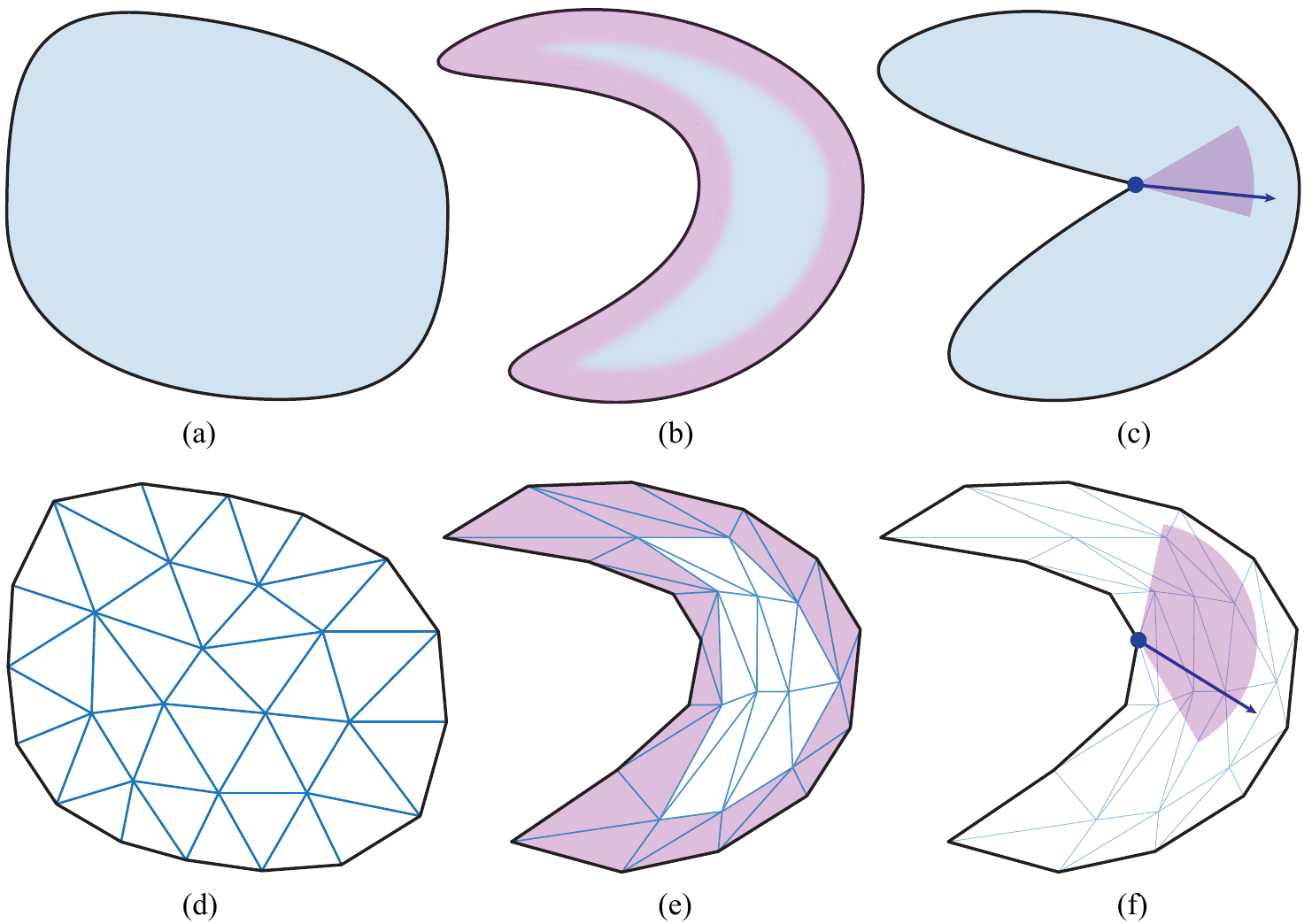}
  \caption{An overview of the results discussed in this paper, illustrating the types of target domains (blue) and boundary conditions (pink) involved in characterizing invertible harmonic maps in the continuous setting (top) and the discrete one (bottom). If the target domain's boundary is convex, the RKC theorem (a) and Tutte embedding theorem (d) provide a guarantee of invertibility. In case the boundary is non-convex, Alessandrini-Nesi characterize invertibility in terms of the map's orientation along the boundary (b). We introduce a simple \emph{cone condition} at non-smooth points of the boundary allowing us to extend their result to the case of a piecewise smooth boundary (c). A discrete version of the cone condition enables the characterization of invertible piecewise-linear maps in the discrete case (f); lastly, we use this result to prove a discrete analog of the Alessandrini-Nesi theorem, characterizing invertibility in terms of the orientation of triangles adjacent to the boundary (e).}
  \label{fig:intro_summary}
\end{figure}

In the continuous case, we build on results by Alessandrini and Nesi \cite{alessandrini2009invertible,alessandrini2017errata}. They show that a sufficient and necessary condition for a harmonic map to be a homeomorphism is that it is a local homeomorphism along the boundary. If the boundary map is differentiable, they prove that a harmonic map is a homeomorphism if and only if it is orientation-preserving along the boundary. We extend their result to the continuous case of a boundary curve admitting a finite number of isolated singularities. To that end, we replace their orientation condition and define a simple geometric condition on the ``forces'' at the non-smooth points of the boundary, which we call the \emph{cone condition}. We then prove that the cone condition characterizes homeomorphic harmonic maps. 

In the discrete case, discrete-harmonic maps into non-convex domains have been characterized in Gortler et al.~\cite{gortler2006discrete}. We provide an alternative characterization in terms of a discrete analog of the cone condition, which, as in the continuous case, is an intuitive geometric condition along the boundary. We prove that it suffices to consider the discrete cone condition at the ``non-convex boundary'' (\ie, reflex) vertices to characterize discrete-harmonic intersection-free embeddings of triangulations. These, in turn, induce piecewise-linear homeomorphisms. Finally, we derive a discrete analog of the result by Alessandrini and Nesi, characterizing discrete-harmonic maps in terms of the orientation of triangles adjacent to the boundary. 

The rest of the paper is organized as follows: Section~\ref{sect:cont} is concerned with harmonic maps in the continuous case and Section~\ref{sect:discrete} with discrete-harmonic maps in the case of triangulations. The analogy between the conditions and results for the continuous and discrete statements is illustrated in Figure~\ref{fig:intro_summary}.

%%%%%%%%%%%%%%%%%%%%%%%%%%%%%%%%%%%%%%%%%%%%%%%%%%%%%%%%%%%%%%%%%%%%%%%%%

\section{Continuous Harmonic Mappings}
\label{sect:cont}

\subsection{Preliminaries}
Let $\DISK=\set{(x,y) : x^2 + y^2 < 1}$ denote the open unit disk and let $\BMAP:\partial \DISK \to \BND$ be a homeomorphism from the unit circle onto a simple closed curve $\BND$ enclosing a bounded set $\IMG \subset \Real^2$. We consider the harmonic extension of $\BMAP$, that is the map $\MAP:\DISK \to \Real^2$ given by the solution of the 2-dimensional Dirichlet problem
\begin{equation} \label{eqn:cont_harmonic} 
    \begin{cases}
        \Delta \MAP = 0 \qquad &\mbox{in}~\DISK\\
        \MAP = \BMAP \qquad &\mbox{on}~\partial\DISK.
    \end{cases}
\end{equation}
Note that $\MAP = \parr{\MAP_1,\MAP_2}$ has two coordinates, as does $\BMAP = \parr{\BMAP_1,\BMAP_2}$; hence the notation of \eqref{eqn:cont_harmonic} should be understood as being applied to each coordinate separately.

When needed, we shall use the polar coordinates defined by $(\nu,\theta) \mapsto (1-\nu)e^{i\theta}$, with $-\pi \leq\theta\leq \pi$ and $0\leq\nu\leq1$. Note that, as opposed to the classic definition of polar coordinates, in our definition  $\nu = 0 $ corresponds to the boundary of the disc, and $\nu = 1$ corresponds to its center, $\parr{0,0}$. This slight variation will simplify the subsequent equations. Figure~\ref{fig:cont_setup} depicts the problem setup and notations.

\begin{figure}[h!]
  \centering
  \includegraphics[width=0.9\columnwidth]{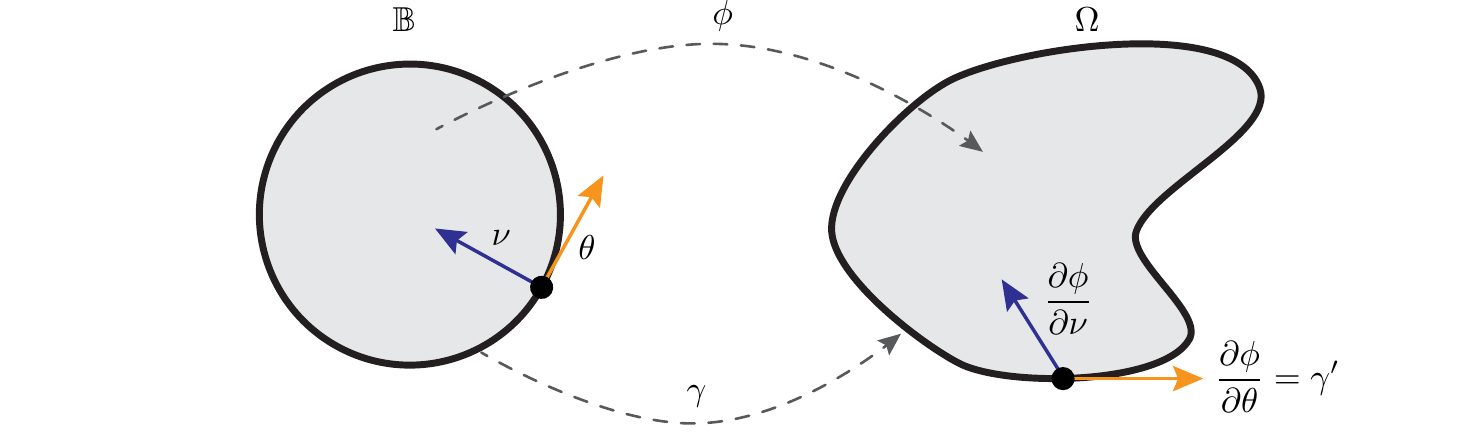}
  \caption{Harmonic extension: problem setup and notations.}
  \label{fig:cont_setup}
\end{figure}

One of the most well-known results on injectivity of harmonic mappings is due to Rad\'o, Kneser and Choquet (RKC), who prove that if $\IMG$ is convex then the harmonic extension of $\BMAP:\partial \DISK \to \partial \IMG$ is a homeomorphism \cite{duren2004harmonic, rado1926aufgave, kneser1926losung, choquet1945type}:

\begin{theorem}[Rado-Kneser-Choquet] \label{cont:RKC}
If $\IMG$ is convex, then $\MAP$ is a homeomorphism of $\overline{\DISK}$ onto $\overline{\IMG}$. 
\end{theorem}

The RKC theorem is known to be sharp, in the following sense:
\begin{theorem} [Choquet \cite{choquet1945type}] \label{thm:Choquet}
If $\IMG$ is not convex, there exists a homeomorphism $\BMAP:\partial \DISK \to \partial \IMG$ such that the harmonic extension $\MAP$ is not a homeomorphism. 
\end{theorem}
An alternative proof of Theorem~\ref{thm:Choquet}, with a more explicit construction, was recently given by Alessandrini and Nesi \cite{alessandrini2009invertible}. We provide an even simpler proof in Section~\ref{sect:alt_proof_choquet}.

On the other hand, even when $\partial \IMG$ encloses a non-convex domain, there always exist boundary maps whose harmonic extension, \ie, the solution  to \eqref{eqn:cont_harmonic} $\MAP$,  is a homeomorphism. In fact, for any  simple closed curve $\BND = \partial \IMG$, there always exists a specific choice of the boundary map $\BMAP:\partial \DISK \to \partial \IMG$ that yields a homeomorphic $\MAP$: one such boundary map can be constructed by considering the Riemann mapping \cite[p.~420]{palka1991introduction} from the interior of the unit disk to the interior of the domain enclosed by $\partial \IMG$, which is a harmonic homeomorphism. Then, by Caratheodory's theorem \cite[p.~445]{palka1991introduction}, there is a continuous extension to the boundary, yielding the desired $\gamma$. 

The following theorem, by Alessandrini and Nesi, states specific conditions that ensure that the harmonic extension $\phi$ is injective even when the target boundary is non-convex, assuming the boundary map is $C^1$. 

\begin{theorem}[Alessandrini--Nesi \cite{alessandrini2009invertible,alessandrini2017errata}] \label{thm:cont_AN} Let $\BMAP:\partial \DISK \to \partial \IMG$ be an orientation preserving diffeomorphism of class $C^1$ onto a simple closed curve $\BND = \partial \IMG$ and let $\MAP \in C^2(\DISK, \Real^2) \cap C^1(\overline{\DISK}, \Real^2)$ be a solution of \eqref{eqn:cont_harmonic}. 

The mapping $\MAP$ is a diffeomorphism of $\overline{\DISK}$ onto $\overline{\IMG}$ if and only if
\begin{equation*}
\det D\MAP > 0 \quad \mbox{everywhere on}~\partial \mathbb{\DISK}.
\end{equation*}
\end{theorem}

\subsection{Harmonic extensions to non-smooth boundary maps} 
As a step towards the discrete case, we consider also mappings $\BMAP:\partial \DISK \to \BND$ that are only \emph{piecewise} differentiable, and aim to formulate a criterion for guaranteeing the harmonic extension $\MAP$ is a homeomorphism.

Towards that end, we draw inspiration from Alessandrini and Nesi (Theorem~\ref{thm:cont_AN}), but  replace their determinant condition with a weaker condition, expressed in terms of one-sided derivatives, thus extending their result to the case of a piecewise-smooth boundary. We will assume that the derivative $\BMAP' = {d\BMAP}/{d\theta}$ of the boundary map exists and is strictly bounded away from zero in all but finitely many points. We denote the set of boundary points in which the derivative does not exist as $\CUSPS \subset \partial \DISK$. We will also assume that the normal derivative $\partial \MAP / \partial \nu$ of the harmonic extension is defined at all points in $\partial\DISK$, does not vanish, and is continuous on $\partial \DISK$. For points $e^{i\theta} \in \CUSPS \subset \partial \DISK$ where the derivative $\BMAP' = {d\BMAP}/{d\theta}$ does not exist, we assume that the one-sided derivatives $\BMAP'_+ $ and $\BMAP'_-$ exist and do not vanish.

\begin{figure}[t!]
  \centering
      \includegraphics[width=.9\columnwidth]{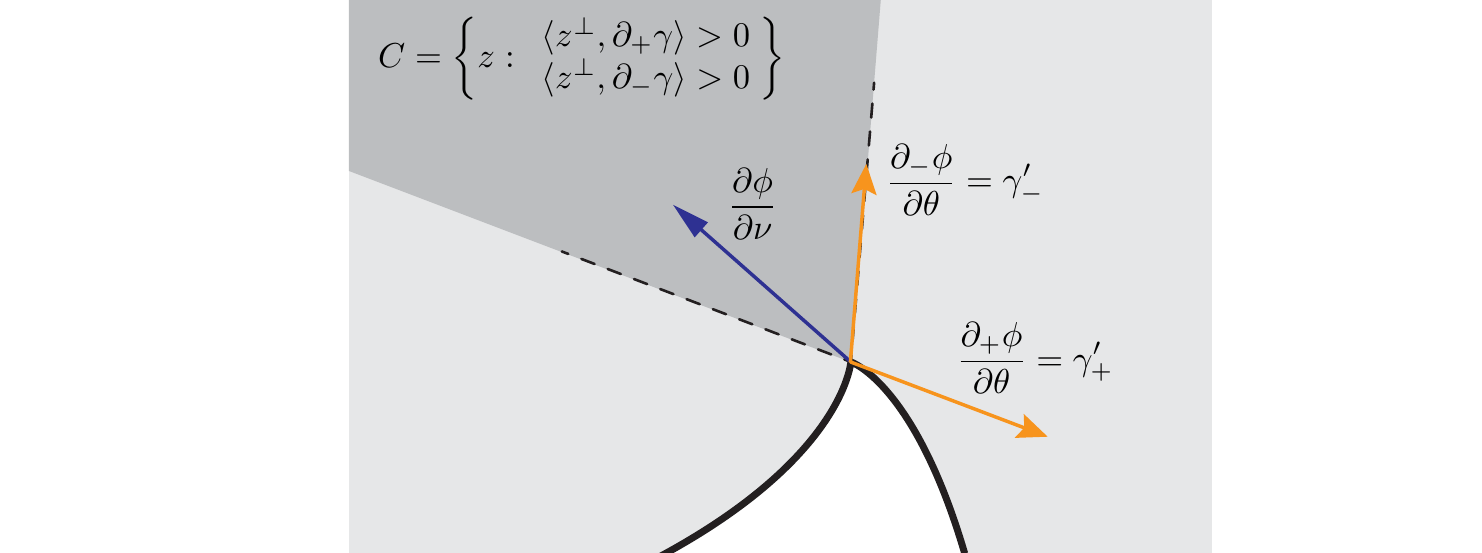}
  \caption{The \emph{cone condition}: the normal derivative is contained within the intersection of the ``inward-pointing'' half-planes supporting the one-sided derivatives.}
  \label{fig:cont_cone}
\end{figure}

For every boundary point $e^{i\theta} \in\partial\DISK$ we define a \emph{cone} $\CC_\theta$ as the set resulting from the intersection of the two ``inward-pointing'' half-planes supporting the vectors of the one-sided derivatives $\BMAP'_+$ and $\BMAP'_-$; see figure~\ref{fig:cont_cone} for an illustration in which the cone is colored in dark-grey.

Formally, suppose  $\BMAP$ is an orientation-preserving homeomorphism of the boundary, \ie, it traverses $\BND$ in the counter-clockwise direction. The open cone $\CC_\theta$ at point $e^{i\theta}$ is defined as
\begin{equation} \label{eqn:cont_cone}
    \CC_\theta = 
    %\set{z\in \Real^2 : z \times \partial_-\BMAP\parr{\theta}>0 \mbox{ and } z \times \partial_+\BMAP\parr{\theta}>0}
    \set{z\in \Real^2 : \ip{z^\perp,\BMAP'_+\parr{e^{i\theta}}}>0 \mbox{ and } \ip{z^\perp,\BMAP'_-\parr{e^{i\theta}}}>0},
\end{equation}
where $z^\perp$ denotes the $\pi/2$ clockwise rotation of $z$. 

\pagebreak
\begin{wrapfigure}[10]{r}{0.13\columnwidth}
\vspace{-6pt}
\hfill
\includegraphics[width=0.10\columnwidth]{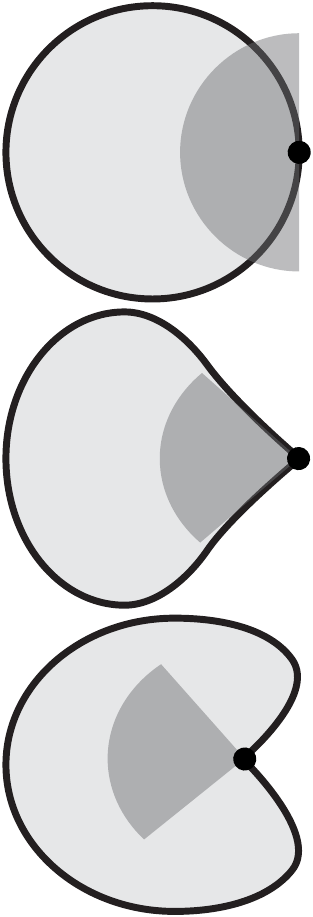}
\end{wrapfigure}
Note that for a differentiable boundary point, \ie, $e^{i\theta} \notin \CUSPS$, the two one-sided derivatives are equal,  $\BMAP'_+\parr{e^{i\theta}} = \BMAP'_-\parr{e^{i\theta}}$. Thus, in this case, $\CC_\theta$ is simply a half-space. On the other hand, if the one-sided derivative point in opposite directions (in which case the opening angle is $0$ or $2\pi$) then the cone $\CC_\theta$ is empty. Lastly, note that since $\CC_\theta$ is an intersection of two half-spaces, it has an opening angle of at most $\pi$, regardless of the opening angle of $\partial \IMG$ at that point (see inset).

Given a map $\phi:\DISK\to\IMG$ we say it satisfies the \emph{cone condition} at $e^{i\theta} \in \partial \DISK$ if the following equation holds
\begin{equation}
\label{eqn:cone_cond}
\frac{\partial \MAP}{\partial \nu}\parr{e^{i\theta}} \in \CC_\theta.
\end{equation}
Intuitively, the cone condition requires that the derivative of the harmonic map in the normal direction is contained within the cone.

It turns out that this simple condition can be used to formulate a theorem similar to that of Alessandrini-Nesi's, in which the cone condition replaces the determinant condition, ensuring that the resulting map $\phi$ is a homeomorphism into $\IMG$. Furthermore, we will show that, as the cone condition permits the boundary map $\gamma$ to have derivative discontinuities, it extends Alessandrini-Nesi's result to harmonic maps into domains with a non-smooth boundary. More precisely,
\begin{restatable}{theorem}{contmain}
%\begin{theorem} 
\label{thm:cont_main}
Assume that $\BMAP:\partial \DISK \to \BND$ is an orientation preserving homeomorphism onto a simple closed curve $\BND = \partial \IMG$ that is $C^1$ at all but a finite set of points $\CUSPS \subset \partial \DISK$. Also assume that $\BMAP'$ does not vanish on $\partial \DISK \setminus \CUSPS$, and that for each point in $\CUSPS$ the one-sided derivatives $\BMAP'_+ $ and $\BMAP'_-$ exist, do not vanish and equal the limit of $\BMAP'$ on the corresponding side. 

Let $\MAP \in C^2(\DISK, \Real^2) \cap C^1(\overline{\DISK} \setminus \CUSPS, \Real^2)$  be a solution of \eqref{eqn:cont_harmonic}. Further assume that everywhere on $\partial \DISK$, the normal derivative $\partial \MAP/\partial \nu$ exists, is continuous and does not vanish. The map $\MAP$ is a homeomorphism of $\overline{\DISK}$ onto $\overline{\IMG}$ if and only if it satisfies the cone condition \eqref{eqn:cone_cond} everywhere on $\partial \DISK$.
\end{restatable}
%\end{theorem}

\begin{proof}
A main ingredient of our proof is another theorem by  Alessandrini and Nesi: 
\begin{theorem}[\cite{alessandrini2009invertible}, Theorem 1.7] \label{thm:an_local}
Let $\BMAP: \partial \DISK \to \BND \subset \Real^2$ be a homeomorphism onto a simple closed curve $\BND = \partial \IMG$ and let $\MAP \in W_{\textrm{loc}}^{1,2}(\DISK, \Real^2) \cap C(\overline{\DISK}, \Real^2)$ be the solution of \eqref{eqn:cont_harmonic}. 

The mapping $\MAP$ is a homeomorphism of $\overline{\DISK}$ onto $\overline{\IMG}$ if and only if, for every point $e^{i\theta} \in \partial \DISK$, the mapping $\MAP$ is a local homeomorphism at $e^{i\theta}$ (\ie, there exists a neighborhood of $e^{i\theta}$ where $\MAP$ is injective). 
\end{theorem}

One direction of the proof (``if'') can thus be obtained by showing that for every boundary point $e^{i\theta} \in \partial \DISK$, the cone condition implies that $\MAP$ is a local homeomorphism at $e^{i\theta}$. %(Recall that we use polar coordinates defined as $(\nu,\theta) \mapsto \parr{(1-\nu)\cos(\theta), (1-\nu)\sin(\theta)}$, thus $\nu=0$ corresponds to the boundary.)

The cone condition $\frac{\partial \MAP}{\partial \nu} \parr{e^{i\theta}} \in \CC_\theta$ implies that
\begin{equation} \label{eqn:cont_innerprod}
    \ip{\Bigl(\frac{\partial \MAP}{\partial \nu}\Bigr)^\perp, \frac{\partial_\pm \MAP}{\partial \theta}} > 0
\end{equation}
at the boundary point $e^{i\theta} \in \partial \DISK$.

For differentiable boundary points $e^{i\theta} \in \partial \DISK \setminus S$, \eqref{eqn:cont_innerprod} along with continuity imply that the normal and tangential derivatives are linearly independent in a neighborhood of $e^{i\theta}$. Hence, the differential of the map $D\MAP$ is invertible and, by the inverse function theorem, there exists a neighbourhood of $e^{i\theta}$ in which $\MAP$ is a local homeomorphism.

For a point $e^{i\theta} \in S \subset \partial \DISK$, where the derivative is discontinuous, the argument is more subtle, but follows a similar idea. In this case, the cone condition \eqref{eqn:cont_innerprod} implies that the boundary map $\BMAP:\partial \DISK \to \BND \subset \Real^2$ is locally monotone in the direction perpendicular to the normal derivative at $e^{i\theta}$; see Figure~\ref{fig:cone_monotinicity} for an illustration. 
\begin{figure}[t!]
  \centering
      \includegraphics[width=.9\columnwidth]{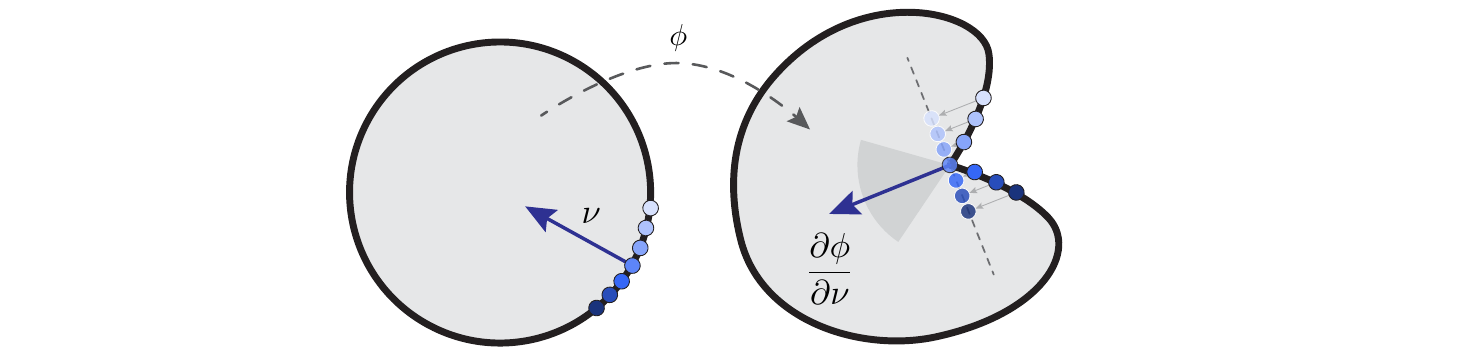}
  \caption{Local monotonicity: The cone condition $\frac{\partial \MAP}{\partial \nu} \in \CC_\theta$ implies that the boundary map $\BMAP:\partial \DISK \to \BND$ is locally monotone in the direction perpendicular to the normal derivative $\frac{\partial \MAP}{\partial \nu}$.}
  \label{fig:cone_monotinicity}
\end{figure}
Namely, there exists a $\delta$-neighborhood of $\theta$ and $c>0$ such that for all $s<t$ in $\brac{\theta-\delta,\theta+\delta}$ we have
\begin{equation} \label{eqn:bnd_tan_monotone}
    \ip{\Bigl(\frac{\partial \MAP}{\partial \nu}\Bigr)^\perp, \BMAP\parr{e^{it}}-\BMAP\parr{e^{is}}} \geq c\parr{t-s}.
\end{equation}
(With this notation we assume, without loss of generality, that $\abs{\theta}$ is sufficiently far away from $\pi$.) %In Appendix~\ref{sect:apndx_proof_cont} we
One can leverage this monotonicity to establish a uniform bound on the angle between the normal and tangential derivatives of the harmonic extension, $\partial \MAP/\partial \nu$ and $\partial \MAP/\partial \theta$, in a neighborhood of $e^{i\theta}$. Thus, the cone condition implies that, in a sufficiently small neighborhood of $e^{i\theta}$, the normal and tangential derivatives are well-behaved and define a consistent local coordinate system. More precisely, we prove (following this line of reasoning) in Appendix~\ref{sect:apndx_proof_cont} the following Lemma:
%\begin{lemma} 
\begin{restatable}{lemma}{contbndlocalhomeo}
\label{lemma:cont_bnd_local_homeo}
Under the conditions of Theorem~\ref{thm:cont_main}, the harmonic extension $\MAP$ is a local homeomorphism around all boundary points $e^{i\theta}$, $-\pi \leq \theta \leq \pi$.
\end{restatable}
%\end{lemma}
Thus we see that if the cone condition holds for all boundary points $\partial \DISK$
then (since $\MAP$ is a local-homeomorphism at all points on $\partial\DISK$) Theorem~\ref{thm:an_local} implies $\MAP$ is a homeomorphism of $\DISK$ onto $\IMG$. 

For the other direction (``only if''), if $e^{i\theta} \in \partial \DISK$ does not satisfy the cone condition then the inner product in \eqref{eqn:cont_innerprod} is either zero or negative, for at least one of the one-sided derivatives; without loss generality we shall assume the cone condition is violated for $\partial_+ \MAP / \partial \theta$. If \eqref{eqn:cont_innerprod} vanishes then $\partial \MAP / \partial \nu$ and $\partial_+ \MAP / \partial \theta$ are colinear, in which case Theorem~\ref{thm:an_local} implies that $\MAP$ is not a homeomorphism. Otherwise, if \eqref{eqn:cont_innerprod} is negative, then a neighborhood of $e^{i\theta} \in \partial \DISK$ in $\DISK$ is mapped outside of $\overline{\IMG}$, as well implying that $\MAP$ is not a homeomorphism.
\end{proof}

%%%%%%%%%%%%%%%%%%%%%%%%%%%%%%%%%%%%%%%%%%%%%%%%%%%%%%%%%%%%%%%%%%%%%%%%%

\section{Discrete-Harmonic Mappings}
\label{sect:discrete}

In this section we discuss discrete analogs of the results presented in Section~\ref{sect:cont}.

\subsection{Preliminaries} \label{sect:prelim}

\paragraph{Triangulations.}
We consider a triangulation $\TT=(V,E,F)$ over a finite set of vertices $V = \set{v_1,\dots,v_n}$ defined by a set of faces (triangles)
$$
F \subseteq \set{\set{v_i,v_j,v_k}: v_i,v_j,v_k\in V \textrm{~are distinct}},
$$
and a corresponding set of edges
$$
E = \set{\parr{v_i,v_j}\in V^2:\set{v_i,v_j,v_k}\in F \textrm{~for some~}v_k\in V}.
$$
We think of the set of faces $F$ as (non-ordered) subsets of vertex triplets. The edges $E$ are \emph{ordered pairs} of vertices, corresponding to the directed edges connecting the vertices of each face.

A \emph{boundary edge} of the triangulation is an edge associated with only a single triangle; that is, $\parr{v_i,v_j}\in E$ is a boundary edge if there exists a unique $v_k\in V$ with $\set{v_i,v_j,v_k}\in F$. The boundary of the triangulation is the union of all boundary edges. We say that a vertex is a \emph{boundary vertex} if it belongs to a boundary edge and an \emph{interior vertex} otherwise. We say that a triangulation is \emph{3-connected} if it remains connected after the removal of any two vertices and their incident edges. 

A \emph{drawing} of a given triangulation $\TT$ is a mapping of each vertex $v_i\in V$ to a distinct point $x_i\in\Real^2$ of the plane and of each edge $(v_i,v_j)\in E$ to a simple curve with endpoints $x_i$ and $x_j$. A drawing is \emph{intersection-free} if its edges do not intersect except, possibly, at common endpoints. We say that a drawing is \emph{proper} if its faces coincide with the faces of $\TT$ and the edges of its external (unbounded) face coincide with the boundary edges of $\TT$. A triangulation $\TT$ has \emph{disk topology} if its boundary is a cycle and it has a proper intersection-free drawing. See Figure~\ref{fig:proper_drawings} for illustrations. 

\begin{figure}[t!]
  \centering
  \includegraphics[width=.9\columnwidth]{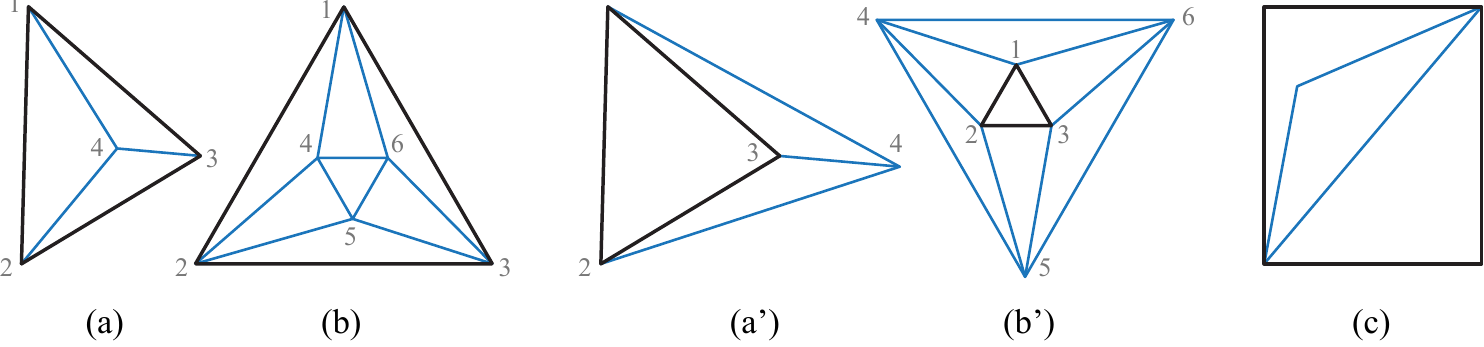}
  \caption{(a)-(b) illustrate proper intersection-free  drawings of two simple triangulations with disk topology; boundary edges are colored black and interior edges blue. (a')-(b') demonstrate non-proper realizations of the same triangulations, respectively; vertices are numbered correspondingly. (c) is an intersection-free drawing of a triangulation that does not have disk topology (note that the diagonal edge is associated with three triangles). }
  \label{fig:proper_drawings}
\end{figure}

Lastly, a \emph{straight-line drawing} is a drawing in which each edge is realized by a straight line segment connecting its endpoints; we use a pair $\parr{\x,\TT}$ to denote the straight-line drawing determined by an embedding $\x=\parr{x_1,...,x_n}\in \Real^{2 \times n}$. If $\parr{\x,\TT}$ is a proper intersection-free straight-line drawing we denote by $\Omega_{\parr{\x,\TT}} \subset \Real^2$ the simple polygonal domain enclosed by the boundary of the the straight-line drawing $\parr{\x,\TT}$. See Figure~\ref{fig:straight_line_drawing} for illustrations. 

\begin{figure}[t!]
  \centering
  \includegraphics[width=.9\columnwidth]{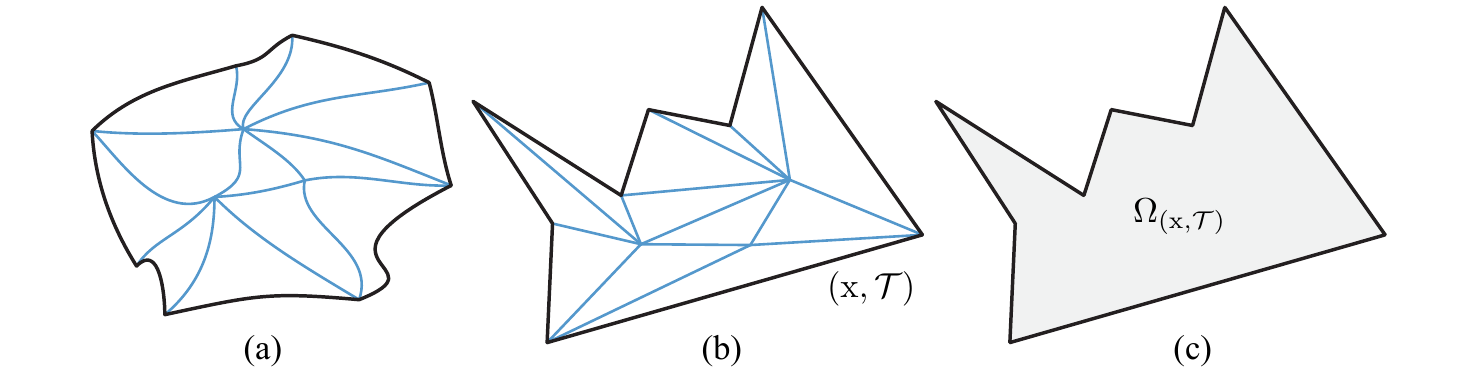}
  \caption{(a) Intersection-free drawing of a triangulation $\TT$. (b) An intersection-free straight-line drawing $\parr{\x,\TT}$ of the same triangulation. (c) The simple polygonal domain $\Omega_{\parr{\x,\TT}}$ enclosed by the boundary of the the straight-line drawing $\parr{\x,\TT}$.}
  \label{fig:straight_line_drawing}
\end{figure}

\paragraph{Discrete-harmonic embeddings.} 
Moving forward, we assume that we are given a proper intersection-free straight-line drawing $\parr{\x,\TT}$ of a 3-connected triangulation $\TT$ with disk topology and vertex coordinates $\x=\parr{x_1,...,x_n}\in \Real^{2 \times n}$. 

Let $\y=\parr{y_1,...,y_n}\in \Real^{2 \times n}$ be some assignment of vertex coordinates and let $w=\parr{w_{ij}}$ be weights associated with the directed edges $(v_i,v_j) \in E$ of $\TT$; note that we do not assume that $w_{ij}=w_{ji}$. We define the \emph{discrete Laplace} operator $\Lap_{w}(\y,\TT):V\to \Real^2$ with weights $w$ with respect to the straight-line drawing $(\y,\TT)$ by
\begin{equation} \label{eqn:discrete_laplace}
\brac{\Lap_{w}(\y,\TT)}(v_i) = \sum_{j \in \NN_\TT(v_i)} w_{ij} \parr{y_j - y_i},
\end{equation}
where $\NN_\TT(v_i)$ is the set of indices of the neighbors of $v_i\in V$ in $\TT$. Namely, $\brac{\Lap_{w}(\y,\TT)}(v_i)$ is the weighted sum of the edges incident to the $i$'th vertex in the straight-line drawing $(\y,\TT)$ of the triangulation $\TT$.

Let $\PP\subset\Real^2$ a simple polygonal domain with number of vertices corresponding to the boundary of $\TT$. We say that an assignment $\y=\parr{y_1,...,y_n}\in \Real^{2 \times n}$ of vertex coordinates is \emph{discrete-harmonic} into $\PP$ with weights $w=\parr{w_{ij}}$ if it satisfies
\begin{equation} \label{eqn:discrete_harmonic}
\Lap_{w}(\y,\TT) = 0
\end{equation}
for all interior vertices and maps the boundary vertices of $\TT$ onto the boundary of $\PP$.

For an assignment $\y=\parr{y_1,...,y_n}$ of vertex coordinates, we denote by $\phi_{\parr{\x,\TT}\to\parr{\y,\TT}}:\Omega_{\parr{\x,\TT}} \to \Real^2$ the continuous piecewise-linear map defined by linearly interpolating the vertex assignments $\phi(x_i)\mapsto y_i$ over each triangle.

\subsection{The Convex Case}
Discrete-harmonic embeddings are often discussed in the case that $\PP$ is convex, following a classic result originating in the work of Tutte \cite{tutte1963draw}:
\begin{theorem}
\label{thm:TutteNEW}
Let $\y=\parr{y_1,...,y_n}\in \Real^{2 \times n}$ be a discrete-harmonic embedding \eqref{eqn:discrete_harmonic} into a simple polygonal domain $\PP$. If $\PP$ is convex then for any positive weights $w_{ij}>0$ the straight-line drawing $\parr{\y,\TT}$ is intersection-free. Moreover, the piecewise-linear map $\phi_{\parr{\x,\TT}\to\parr{\y,\TT}}$ induced by $\y$ is a homeomorphism of $\Omega_{\parr{\x,\TT}}$ onto $\PP$.
\end{theorem}
The result proved in \cite{tutte1963draw} addresses the intersection-free realization of a more general class of planar 3-connected graphs via the solution of equations similar to \eqref{eqn:discrete_harmonic} with uniform weights. The case of arbitrary positive weights is similar, and is discussed in the context of triangular meshes in \cite{gortler2006discrete}. Floater \cite{floater2003one} shows that the discrete-harmonic piecewise-linear map $\phi_{\parr{\x,\TT}\to\parr{\y,\TT}}$ is a homeomorphism. This result is often seen as a discrete analog of the Rado-Kneser-Choquet Theorem.

\begin{figure}[b!]
  \centering
  \includegraphics[width=.9\columnwidth]{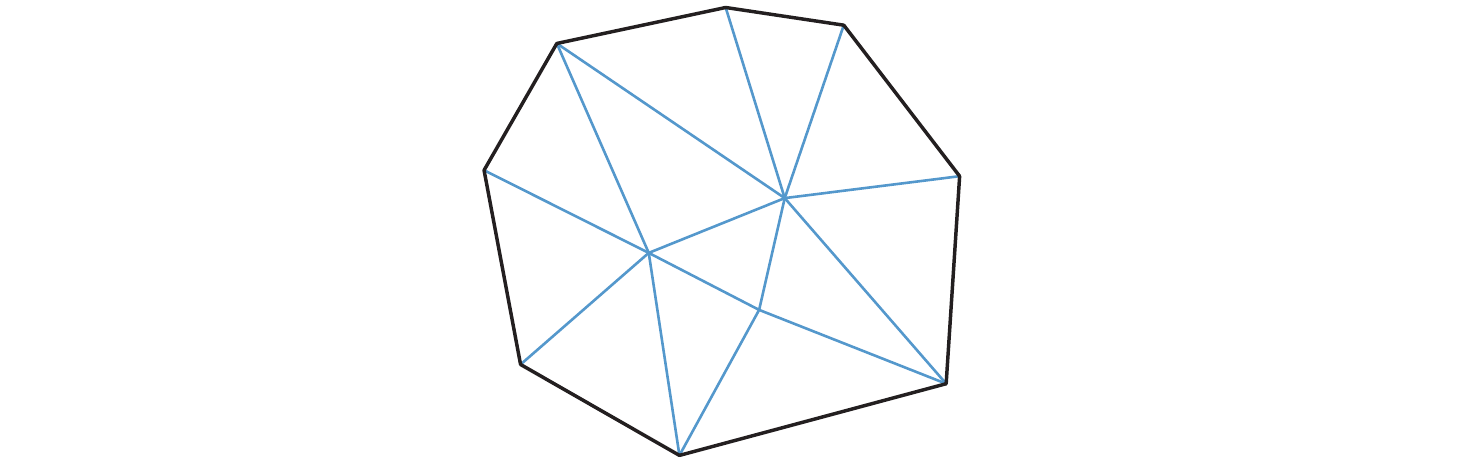}
  \caption{Tutte embedding: discrete-harmonic embedding of a 3-connected triangulation with a convex boundary. Theorem~\ref{thm:TutteNEW} guarantees that this straight-line drawing is intersection-free. }
  \label{fig:tutte_embedding_convex}
\end{figure}

\pagebreak
\subsection{The Non-Convex Case}

When $\PP$ is convex, Theorem~\ref{thm:TutteNEW} guarantees an intersection-free straight-line drawing for any choice of positive weights $w_{ij}>0$. This guarantee, however, fails to hold whenever $\PP$ is non-convex. While certain choices of weights could lead to an intersection-free drawing, other choices might result in intersections; in fact, a choice of weights that produces an intersection-free drawing might not even exist. See Figure~\ref{fig:tutte_nonconvex_fails} for an illustration. Next, we focus on the non-convex case and propose a simple geometric condition that, when satisfied, provides a similar guarantee.

\begin{figure}[t!]
  \centering
  \includegraphics[width=0.9\columnwidth]{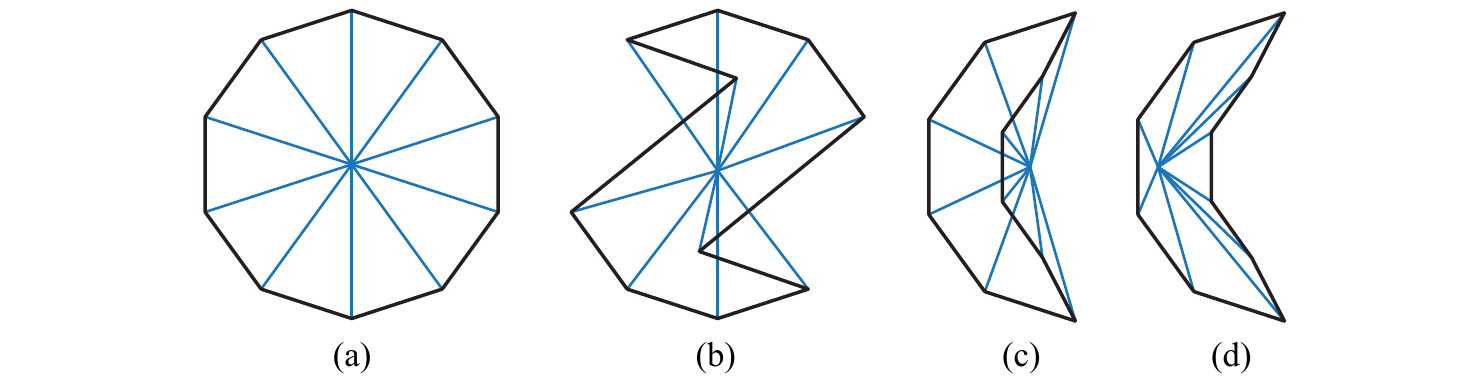}
  \caption{Discrete-harmonic embedding is intersection free if the boundary is convex (a). If the boundary is non-convex, an intersection free embedding might not exist (b). (c) and (d) are discrete-harmonic with the same boundary but different choices of positive weights $w_{ij}>0$; some choice of weights result in intersections (c) whereas others produce an intersection-free drawing (d).}
  \label{fig:tutte_nonconvex_fails}
\end{figure}

\paragraph{Boundary cones.}
We begin by associating a cone with the vertices of a simple polygon $\PP$. For a vertex $p$ of $\PP$ we define the cone $\CC_\PP$ at $p$ as the intersection of the two inward-pointing open half-planes supporting the edges incident to $p$; see Figure~\ref{fig:boundary_cone} for an illustration. 

Formally, let $p',p''$ be the vertices of $\PP$ adjacent to $p$ and suppose that the triplet $p',p,p''$ traverses the boundary of the simple polygon $\PP$ in counter-clockwise order. We define the two half-planes
$$
H^-(p) = \set{z\in \Real^2 : \ip{z^\perp, p-p'} > 0},
$$
and
$$
H^+(p) = \set{z\in \Real^2 : \ip{z^\perp, p''-p} > 0},
$$
where $z^\perp$ denotes the $\pi/2$ clockwise rotation of $z$. Then the cone $\CC_\PP$ at the vertex $p$ is given by
\begin{equation}
\CC_{\PP}(p) = H^-(p) \cap H^+(p).
\end{equation}
Figure~\ref{fig:boundary_cone} illustrates the cone $\CC_\PP$ for a \emph{convex vertex} whose internal angle is at most $\pi$, and a \emph{reflex vertex} whose internal angle is at least $\pi$. Note that this definition coincides with that of the boundary cone given in Section~\ref{sect:cont} for the continuous case; see Figure~\ref{fig:cont_cone} for an illustration.

\begin{figure}[t!]
  \centering
  \includegraphics[width=.9\columnwidth]{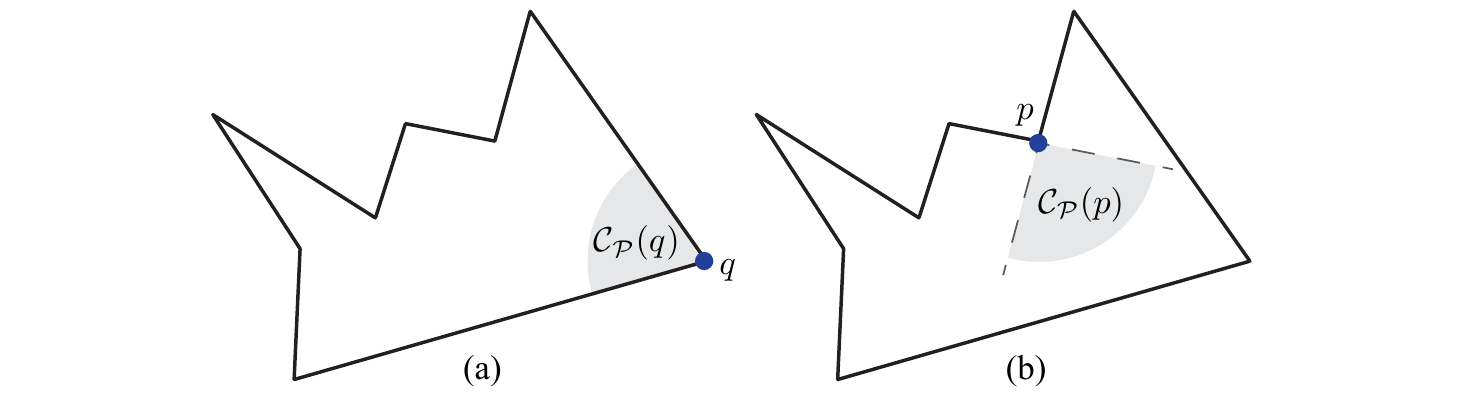}
  \caption{Illustration of the cone $\CC_{\PP}$ for (a) convex vertex $q$ and (b) reflex vertex $p$.}
  \label{fig:boundary_cone}
\end{figure}

\paragraph{Non-convex discrete-harmonic embeddings.}
Suppose $\y=\parr{y_1,...,y_n}\in \Real^{2 \times n}$ is a discrete-harmonic embedding into a simple \emph{non-convex} polygonal domain $\PP$ satisfying \eqref{eqn:discrete_harmonic} with positive weights $w=\parr{w_{ij}}$. 

In this embedding, the boundary vertices $v_i$ of $\TT$ are mapped to vertices of the polygon $\PP$. We shall use $\CC_{\PP}(y_i)$ to denote the boundary cone associated with the embedding $y_i\in\Real^2$ of the $i$-th vertex $v_i$.

We also note that for boundary vertices $\brac{\Lap_{w}(\y,\TT)}(v_i)$ does not necessarily vanish and can be interpreted as the force applied on the constrained boundary vertices by their neighbors. See Figure~\ref{fig:teaser} for example harmonic maps in which $\brac{\Lap_{w}(\y,\TT)}(v_i)$ is illustrated by red arrows along the boundary. As such, it plays a role similar to the normal derivative in the continuous case; in fact, $\brac{\Lap_{w}(\y,\TT)}(v_i)$ is proportional to the gradient of the energy corresponding to the variational form of \eqref{eqn:discrete_harmonic} when it exists.

Our main result is a discrete analog of Theorem~\ref{thm:cont_main} and establishes a connection between the boundary cones $\CC_{\PP}$ and $\Lap_{w}(\y,\TT)$ at the respective boundary vertices. We say that a boundary vertex $v_i$ satisfies the \emph{cone condition} if its embedding satisfies $\brac{\Lap_{w}(\y,\TT)}(v_i)\in \CC_{\PP}(y_i)$. See Figure~\ref{fig:cone_condition}. In turn, we show that satisfying the cone condition for the \emph{reflex boundary vertices}, \ie, vertices mapped to polygon vertices whose internal angle is greater than $\pi$, is sufficient to guarantee an intersection-free embedding.

\begin{figure}[b!]
  \centering
  \includegraphics[width=.9\columnwidth]{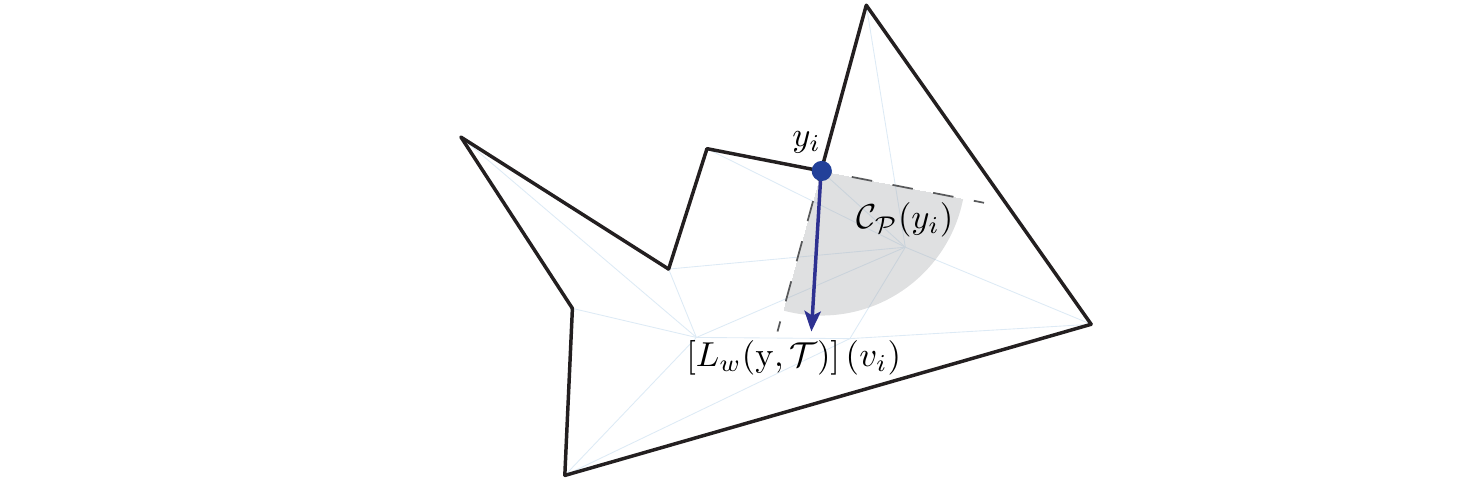}
  \caption{Illustration of the cone condition $\brac{\Lap_{w}(\y,\TT)}(v_i)\in \CC_{\PP}(y_i)$ used in Theorems \ref{thm:nonConvexTutteGRPH} and \ref{thm:nonConvexTuttePWL} to characterize discrete-harmonic embedding into non-convex domains. }
  \label{fig:cone_condition}
\end{figure}

\begin{theorem} \label{thm:nonConvexTutteGRPH}
Let $w_{ij}>0$ be positive weights and let $\y=\parr{y_1,...,y_n}\in \Real^{2 \times n}$ be a discrete-harmonic embedding into a simple polygonal domain $\PP$ satisfying \eqref{eqn:discrete_harmonic}. If the cone condition $\brac{\Lap_{w}(\y,\TT)}(v_i)\in \CC_{\PP}(y_i)$ holds for all reflex boundary vertices $v_i$ of $\TT$ then the straight-line drawing $\parr{\y,\TT}$ is intersection-free.
 %Let $\y=\parr{y_1,...,y_n}\in \Real^{2 \times n}$ be a discrete-harmonic embedding into a simple polygonal domain $\PP$ satisfying \eqref{eqn:discrete_harmonic} with positive weights $w_{ij}>0$. If the cone condition $\brac{\Lap_{w}(\y,\TT)}(v_i)\in \CC_{\PP}(y_i)$ holds for all reflex boundary vertices $v_i$ of $\TT$ then the straight-line drawing $\parr{\y,\TT}$ is intersection-free.
\end{theorem}

The proof provided in Appendix~\ref{sect:proofNonConvexTutteGRPH} is based on reduction to the convex case: we extend the triangulation $\TT$ by adding new edges (but not new vertices) into a triangulation $\TT'$ compatible with $\PP'$, the convex hull of $\PP$; then we show how the cone condition guarantees the existence of positive weights on the edges of the extended triangulation, with which the discrete-harmonic embedding into the convex polygon $\PP'$ reproduces $\y=\parr{y_1,...,y_n}$; finally, Tutte's Theorem~\ref{thm:TutteNEW} implies that the straight-line drawing $\parr{\y,\TT'}$ of the extended triangulation $\TT'$ is intersection-free, and thus so is $\parr{\y,\TT}$. 

We may further consider the piecewise-linear map $\phi_{\parr{\x,\TT}\to\parr{\y,\TT}}:\Omega_{\parr{\x,\TT}} \to \Real^2$ induced by such an assignment $\y=\parr{y_1,...,y_n}$. As a corollary of Theorem~\ref{thm:nonConvexTutteGRPH}, we prove in Appendix~\ref{sect:proofNonConvexTuttePWL} the following characterization of piecewise-linear homeomorphisms between two straight-line drawings of the same triangulation,

\begin{theorem} \label{thm:nonConvexTuttePWL}
A piecewise-linear mapping $\phi_{\parr{\x,\TT}\to\parr{\y,\TT}}$ is a homeomorphism of $\Omega_{\parr{\x,\TT}}$ onto $\PP$ if and only if there exist positive weights $w_{ij}>0$ such that $\y$ is discrete-harmonic into $\PP$ satisfying \eqref{eqn:discrete_harmonic} and the cone condition $\brac{\Lap_{w}(\y,\TT)}(v_i)\in \CC_{\PP}(y_i)$ holds for all reflex boundary vertices $v_i$ of $\TT$.
\end{theorem}

\begin{figure}[t!]
  \centering
  \includegraphics[width=.9\columnwidth]{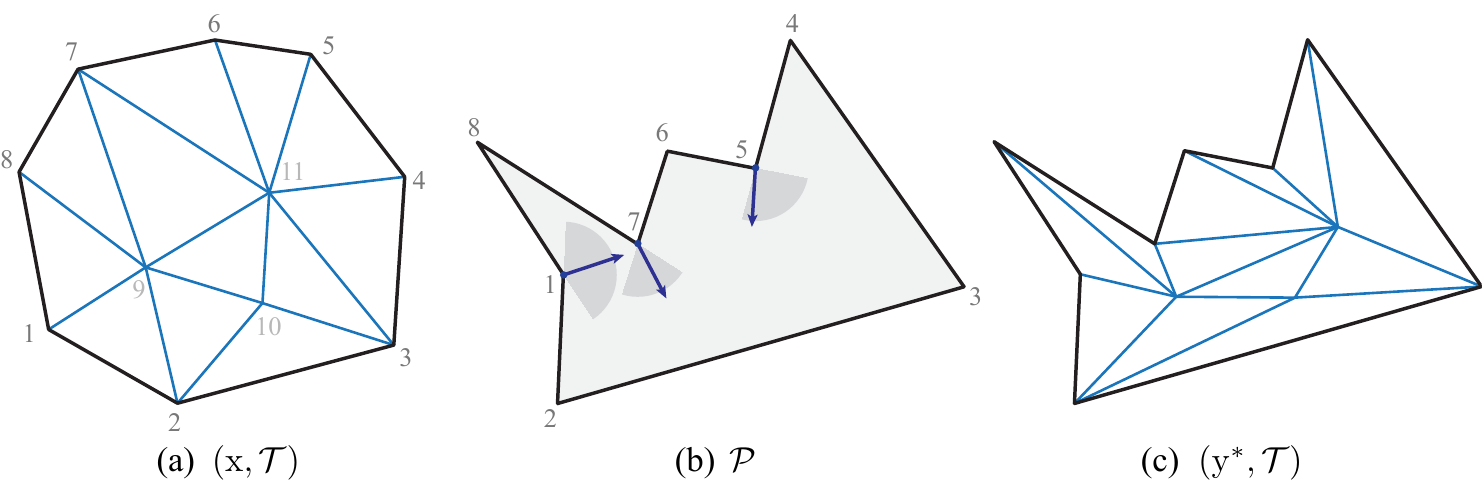}
  \caption{Illustration of Theorems \ref{thm:nonConvexTutteGRPH} and \ref{thm:nonConvexTuttePWL}: The boundary of a given triangulation $\parr{\x,\TT}$ is mapped onto a non-convex polygon $\PP$ and the cone condition is satisfied for all reflex boundary vertices (numbered 1,5 and 7 in the illustration). Theorems \ref{thm:nonConvexTutteGRPH} and \ref{thm:nonConvexTuttePWL} then imply that the straight-line drawing $\parr{y^*,\TT}$ is intersection-free and that $\parr{\x,\TT}$ and $\parr{y^*,\TT}$ are related by a piecewise-linear homeomorphism.}
  \label{fig:nonconvex_tutte}
\end{figure}

Gortler et al.~\cite{gortler2006discrete} also formulate sufficient conditions for invertibility in terms of  reflex vertices. However, their conditions differ from the cone conditions, and the statement and proof presented there follow an index counting argument which does not readily take the form of a discrete analog of the geometric condition of Theorem~\ref{thm:cont_main}. 

The characterization of Theorem~\ref{thm:nonConvexTuttePWL} is provided in terms of the cone condition at the reflex boundary vertices of $\TT$. Theorem~\ref{thm:nonConvexTutteGRPH} can be also used to derive an alternative characterization of homeomorphisms in terms of the differential of the map near the boundary. Note that $\phi_{\parr{\x,\TT}\to\parr{\y,\TT}}$ is piecewise-linear and thus $D\phi_{\parr{\x,\TT}\to\parr{\y,\TT}}$ is a constant $2\times 2$ matrix on the \emph{interior} of every triangle. 
 As a corollary, we can now prove the discrete analog of Theorem~\ref{thm:cont_AN} (Alessandrini–Nesi), due to \cite{gortler2006discrete}:
 
\begin{theorem} \label{thm:dis_AN_PWL}
Let $\y=\parr{y_1,...,y_n}\in \Real^{2 \times n}$ be a discrete-harmonic embedding into a simple polygonal domain $\PP$ satisfying \eqref{eqn:discrete_harmonic} with positive weights $w_{ij}>0$. Further assume that the boundary map is orientation preserving; namely, that the boundary polygons of $\parr{\x,\TT}$ and $\parr{\y,\TT}$ have the same orientation. Then the piecewise-linear map $\phi_{\parr{\x,\TT}\to\parr{\y,\TT}}:\Omega_{\parr{\x,\TT}} \to \Real^2$ is a homeomorphism of $\Omega_{\parr{\x,\TT}}$ onto $\PP$ if and only if 
$$\det D\phi_{\parr{\x,\TT}\to\parr{\y,\TT}}>0$$
on (the interior of) all boundary triangles; \ie, triangles $f\in F$ that have a vertex on the boundary.
\end{theorem}
A proof is provided in Appendix~\ref{sect:proofdis_AN_PWL}.

%%%%%%%%%%%%%%%%%%%%%%%%%%%%%%%%%%%%%%%%%%%%%%%%%%%%%%%%%%%%%%%%%%%%%%%%%

%\matmethods{Please describe your materials and methods here. This can be more than one paragraph, and may contain subsections and equations as required. Authors should include a statement in the methods section describing how readers will be able to access the data in the paper. }

%\showmatmethods{} % Display the Materials and Methods section

%\acknow{Acknowledgments}
%\showacknow{} % Display the acknowledgments section

% Bibliography
\bibliographystyle{unsrt}
\bibliography{refs}

%%%%%%%%%%%%%%%%%%%%%%%%%%%%%%%%%%%%%%%%%%%%%%%%%%%%%%%%%%%%%%%%%%%%%%%%%
\clearpage
\onecolumn
\appendix
\renewcommand{\thesubsection}{\Alph{section}.\arabic{subsection}}

%%%%%%%%%%%%%%%%%%%%%%%%%%%%%%%%%%%%%%%%%%%%%%%%%%%%%%%%%%%%%%%%%%%%%%%%%
%{\titlefont Supplementary Information \par}
%\medskip

\section{Continuous Harmonic Mappings: Proofs}

\subsection{Theorem~\ref{thm:cont_main}: supplemental details}
\label{sect:apndx_proof_cont}

Let us recall Theorem~\ref{thm:cont_main}:
\contmain*

The proof of Theorem~\ref{thm:cont_main} provided in Section~\ref{sect:cont} is complete with the exception of Lemma~\ref{lemma:cont_bnd_local_homeo}, which argues that under the conditions of Theorem~\ref{thm:cont_main}, for every point $e^{i\theta} \in \partial \DISK$, the mapping $\MAP$ is a local homeomorphism at $e^{i\theta}$. In what follows we shall prove this Lemma in a few steps:

\paragraph{A Geometric Lemma.}
To simplify notations and without loss of generality we consider the following setup: We will consider the boundary point $e^{i\theta}|_{\theta=0} \in \partial \DISK$. We further assume it is mapped to $\BND$, the boundary of the target domain, in such a way that 
\begin{equation} \label{eqn:orient_2nd_coordinate}
\frac{\partial}{\partial \nu} \MAP\parr{(1-\nu)e^{i\theta}}|_{\nu=\theta=0} = (-\lambda,0),
\end{equation}
for some $\lambda>0$. In what follows we shall focus on (the more challenging) case where the derivative at $e^{i0}$ is discontinuous; the same proof readily applies to smooth boundary points. With this setup, illustrated in Figure~\ref{fig:cone_monotinicity_reoriented}, the cone condition of \eqref{eqn:cone_cond} implies that the second coordinate of the boundary map $\BMAP$ is monotone in a neighborhood of $e^{i0}$; namely, we have the following Lemma: 
\begin{figure}[h!]
  \centering
      \includegraphics[width=.9\columnwidth]{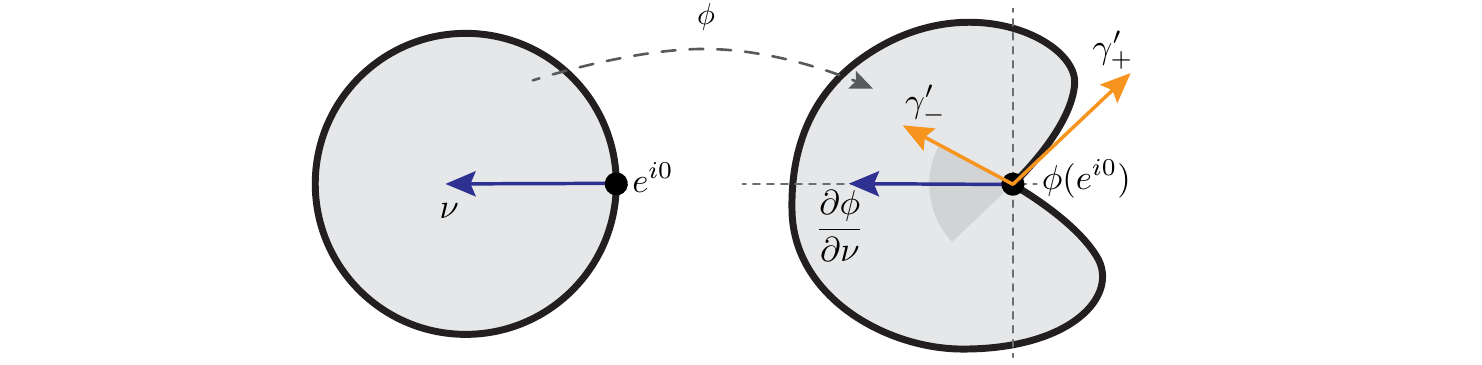}
  \caption{The geometric setup of our proof: without loss of generality, we consider the boundary point $e^{i0}$ and reorient the image such that the normal gradient $\frac{\partial}{\partial \nu}$ points in the negative $x$ direction.}
  \label{fig:cone_monotinicity_reoriented}
\end{figure}
\begin{lemma} \label{lemma:local_monotone_2nd_coordinate}
Denoting $\BMAP = \parr{\BMAP_1,\BMAP_2}$, the cone condition 
\begin{equation} \label{eqn:cone_condition_explicit}
\ip{\Bigl(\frac{\partial \MAP}{\partial \nu}\Bigr)^\perp,\BMAP'_+}>0 \mbox{ and } \ip{\Bigl(\frac{\partial \MAP}{\partial \nu}\Bigr)^\perp,\BMAP'_-}>0
\end{equation}
at $e^{i0}$ implies that for some $\delta>0$ and all $s,t \in \brac{-\delta,\delta}$ with $s<t$ we have
\begin{equation} \label{eqn:bnd_tan_monotone_2nd_coordinate}
\BMAP_2\parr{e^{it}} - \BMAP_2\parr{e^{is}} \geq c\parr{t-s},
\end{equation}
where $c > \min\set{\BMAP'_{+,2},\BMAP'_{-,2}}/2 > 0$ depends on the opening angle of boundary and the magnitudes of the one-sided tangential derivatives.
\end{lemma}
\begin{proof}
Follows from plugging in \eqref{eqn:orient_2nd_coordinate} into \eqref{eqn:cone_condition_explicit} and local continuity.
\end{proof}
Note that Lemma~\ref{lemma:local_monotone_2nd_coordinate} is nothing but the local monotonicity reported previously in \ref{eqn:bnd_tan_monotone}, which can be conveniently expressed in terms of the second coordinate of $\BMAP$ when the normal direction aligns with the $x-axis$ (see Figure~\ref{fig:cone_monotinicity_reoriented}).

\paragraph{An Analytic Lemma.}
\begin{lemma}
Let $f:\partial \DISK \to \Real$ with $\partial \DISK$ parameterized by the (periodized) segment $[-\pi,\pi]$. Let $F:\DISK \to \Real$ be the harmonic extension defined by 
$$
F\parr{re^{it}} = \frac{1}{2\pi} \int_0^{2\pi} 
P_r\parr{t-t'} f(t') dt',
$$
where
$$
P_r\parr{\theta} = \frac{1-r^2}{r^2 + 1 -2r\cos\parr{\theta}}
$$
is the Poisson kernel.

If $f$ is bounded, $\abs{f(t)}\leq M$ for all $t\in\brac{-\pi,\pi}$, and satisfies
$$
f\parr{t}-f\parr{s} \geq c \parr{s-t}
$$
for all $s,t\in\brac{-\delta,\delta}$, then there exists $R<1$ such that for all $r \in (R,1)$ and all $s,t\in\brac{-\delta/8,\delta/8}$ with $s<t$, the harmonic extension $F$ satisfies
$$
F\parr{re^{is}}-F\parr{re^{it}} \geq \frac{c}{2}\parr{s-t}.
$$
\end{lemma}
\begin{proof}
Set $\sigma = s-t$. Note that $0<\sigma<\delta/4$. We have
\begin{align*}
    &F\parr{re^{is}}-F\parr{re^{it}} = 
    \frac{1}{2\pi} \int_{\abs{t'}\leq \pi} 
    P_r\parr{s-t'} \brac{f\parr{t'} - f\parr{t'-\sigma}} dt' \\
    & = 
    \frac{1}{2\pi} \int_{\abs{t'}\leq \delta/3} 
    P_r\parr{s-t'} \brac{f\parr{t'} - f\parr{t'-\sigma}} dt' 
    +
    \frac{1}{2\pi} \int_{\abs{t'} > \delta/3} 
    P_r\parr{s-t'} f\parr{t'} dt' \\
    &\qquad -
    \frac{1}{2\pi} \int_{[-\pi,-\delta/3-\sigma) \cup (\delta/3-\sigma,\pi]} 
    P_r\parr{t-t'} f\parr{t'} dt' \\
    & \geq
    \frac{1}{2\pi} \int_{\abs{t'}\leq \delta/3} 
    P_r\parr{s-t'} c\sigma dt' 
    +
    \frac{1}{2\pi} \int_{[-\pi,-\delta/3-\sigma) \cup (\delta/3-\sigma,\pi]}
    \brac{P_r\parr{s-t'}-P_r\parr{t-t'}} f\parr{t'} dt' \\
    &\qquad +
    \frac{1}{2\pi} \int_{[-\delta/3-\sigma,-\delta/3]}
    P_r\parr{s-t'} f\parr{t'} dt'
    -
    \frac{1}{2\pi} \int_{[\delta/3-\sigma,\delta/3]}
    P_r\parr{t-t'} f\parr{t'} dt' \\
    & =
    c\sigma - \frac{1}{2\pi} \int_{\abs{t'} > \delta/3} 
    P_r\parr{s-t'} c\sigma dt' 
    +
    \frac{1}{2\pi} \int_{[-\pi,-\delta/3-\sigma) \cup (\delta/3-\sigma,\pi]}
    \brac{P_r\parr{s-t'}-P_r\parr{t-t'}} f\parr{t'} dt' \\
    &\qquad +
    \frac{1}{2\pi} \int_{[-\delta/3-\sigma,-\delta/3]}
    P_r\parr{s-t'} f\parr{t'} dt'
    -
    \frac{1}{2\pi} \int_{[\delta/3-\sigma,\delta/3]}
    P_r\parr{t-t'} f\parr{t'} dt' \\
    & = 
    c\sigma - Q_1 + Q_2 + Q_3 - Q_4.
\end{align*}
Next, we will show that each term $Q_j$ satisfies $\abs{Q_j}\leq c\sigma/8$.

For $\abs{t'}>\delta/3$ we have $\abs{t'-s}>\delta/3-\delta/8>\delta/6$ and in turn $\cos\parr{t'-s}\leq\cos\parr{\delta/6}$. Therefore
\begin{align*}
    \abs{Q_1}
    & \leq 
    \frac{c\sigma}{2\pi} \int_{\abs{t'}>\delta/3}
    \frac{1-r^2}{1+r^2 - 2r\cos\parr{\delta/6}} dt' 
    \leq
    c\sigma \frac{(1-r)(1+r)}{(1-r)^2 + 2r(1-cos(\delta/6))}
    \leq
    c\sigma \frac{1/r-1}{1-cos(\delta/6)}.
\end{align*}
Hence, if $R>\brac{1+\frac{1-cos(\delta/6)}{8}}^{-1}$ then for any $r \in (R,1)$
\begin{align*}
    \abs{Q_1}
    & \leq 
    c\sigma (1/R-1)(1-cos(\delta/6))^{-1} \leq c\sigma/8.
\end{align*}

For the $Q_2$ term we use $\abs{f(t)}\leq M$ to get
\begin{align*}
    \abs{Q_2}
    & \leq 
    \frac{M}{2\pi} \int_{[-\pi,-\delta/3-\sigma) \cup (\delta/3-\sigma,\pi]}
    (1-r^2) \abs{
    \frac{1}{1+r^2-2r\cos(s-t')} - \frac{1}{1+r^2-2r\cos(t-t')}
    } dt' \\
    & = 
    \frac{M}{2\pi} \int_{[-\pi,-\delta/3-\sigma) \cup (\delta/3-\sigma,\pi]}
    (1-r^2) \abs{
    \frac{2r\brac{\cos(s-t') - \cos(t-t')}}{
    \brac{(1-r)^2 + 2r(1-\cos(s-t'))}
    \brac{(1-r)^2 + 2r(1-\cos(t-t'))}
    }
    } dt'.
\end{align*}
Note that
\begin{align*}
    \abs{\cos(s-t') - \cos(t-t')} = \abs{2\sin\frac{s+t-2t'}{2} \sin\frac{s-t}{2}} \leq \sigma.
\end{align*}
Also note that on the domain of integration $\abs{t'}>\delta/3$ and therefore both $\cos\parr{s-t'}$ and $\cos\parr{t-t'}$ are less than $\cos\parr{\delta/6}$. In turn, for $r \in (R,1)$ we have
\begin{align*}
    \brac{(1-r)^2 + 2r(1-\cos(s-t'))}
    \brac{(1-r)^2 + 2r(1-\cos(t-t'))} \geq 
    4R^2 \parr{1-\cos\parr{\delta/6}}^2.
\end{align*}
Hence we have
\begin{align*}
    \abs{Q_2}
    & \leq 
    \frac{M}{2\pi} \int_{[-\pi,-\delta/3-\sigma) \cup (\delta/3-\sigma,\pi]}
    (1-r^2) \frac{2r\sigma}{4R^2 \parr{1-\cos\parr{\delta/6}}^2} dt'
    \leq
    \sigma M \frac{1-R}{R}
    \frac{1}{\parr{1-\cos\parr{\delta/6}}^2} dt'.
\end{align*}
If we take $R > \brac{1+\frac{c}{8M}\parr{1-\cos\parr{\delta/6}}^2}^{-1}$ then $\frac{1-R}{R} < \frac{c}{8M}\parr{1-\cos\parr{\delta/6}}^2$, and thus
$\abs{Q_2} \leq c\sigma/8$.

Similarly,
\begin{align*}
    \abs{Q_3}
    & \leq 
    \frac{1}{2\pi} M \sigma
    \frac{1-r^2}{(1-r)^2+2r\parr{1-\cos\parr{\delta/6}}}
    \leq 
    \frac{\sigma M (1-R)}{2\pi R\parr{1-\cos\parr{\delta/6}}} \leq c\sigma/8,
\end{align*}
if $r \in (R,1)$ with $R>\brac{1 + \frac{\pi c}{4M}\parr{1-\cos\parr{\delta/6}}}^{-1}$. Applying the same argument for $Q_4$ shows that $\abs{Q_4} \leq c\sigma/8$ with the same condition on $R$.

Therefore, if we choose
$$
R > \max\set{
\brac{1+\frac{1-cos(\delta/6)}{8}}^{-1},
\brac{1+\frac{c}{8M}\parr{1-\cos\parr{\delta/6}}^2}^{-1},
\brac{1 + \frac{\pi c}{4M}\parr{1-\cos\parr{\delta/6}}}^{-1}
}
$$
then we have for all $r \in (R,1)$ and all $s,t\in\brac{-\delta/8,\delta/8}$ with $s<t$ that 
\begin{equation*}
F\parr{re^{is}}-F\parr{re^{it}} \geq \frac{1}{2} c \sigma  = \frac{1}{2} c \parr{s-t}.\qedhere
\end{equation*}
\end{proof}

\begin{lemma}
Let $f:\partial \DISK \to \Real$ be Lipschitz continuous with Lipschitz constant $L$. Then the harmonic extension $F:\DISK \to \Real$ satisfies
$$
\abs{\frac{\partial}{\partial t} F(re^{it})} \leq L.
$$
\end{lemma}
\begin{proof}
This follows from the definition of the Poisson kernel acting as a convolution.
\end{proof}

\paragraph{Main Lemma.}
We are now ready to prove Lemma~\ref{lemma:cont_bnd_local_homeo}:
\contbndlocalhomeo*

\begin{proof}
Without loss of generality, we consider the normalized setup described above (see Figure~\ref{fig:cone_monotinicity_reoriented}) and show that $\MAP$ is a local homeomorphism in a neighborhood of $e^{i0}$. To that end, we will establish uniform control on the normal and tangential derivatives
$$
\frac{\partial \MAP}{\partial r}\parr{re^{it}} \quad\textrm{and}\quad
\frac{\partial \MAP}{\partial t}\parr{re^{it}}
$$ 
in a neighborhood of $(r,t)=(0,1)$. By assumption we have
$$
\frac{\partial \MAP}{\partial r}\parr{re^{it}} 
= \parr{\lambda,0} + o\parr{\abs{r-1}+\abs{t}}.
$$
Let us write the tangential derivative as
$$
\parr{v_1(re^{it}),v_2(re^{it})} = \frac{\partial \MAP}{\partial t}\parr{re^{it}}.
$$
The previous Lemmas imply that, in a sufficiently small neighborhood, we have that
$$
\abs{v_1(re^{it})} \leq L
$$
and
$$
v_2(re^{it}) \geq \frac{c}{2}.
$$
This shows that
\begin{align*}
    \abs{
    \cos\parr{
    \angle \parr{
    \frac{\partial \MAP}{\partial r}\parr{re^{it}},
    \frac{\partial \MAP}{\partial t}\parr{re^{it}}
    }
    }
    } 
    & =
    \abs{
    \frac
    { \parr{\parr{\lambda,0} + o\parr{\abs{r-1}+\abs{t}}}
    \cdot \parr{v_1(re^{it}),v_2(re^{it})} }
    { \sqrt{\parr{\lambda^2 + o\parr{\abs{r-1}+\abs{t}}}
    \parr{v_1(re^{it})^2 + v_2(re^{it})^2}} }
    } \\
    & =
    \frac{\abs{\lambda v_1(re^{it}) + o\parr{\abs{r-1}+\abs{t}} }}{\sqrt{\parr{\lambda^2 + o\parr{\abs{r-1}+\abs{t}}}
    \parr{v_1(re^{it})^2 + v_2(re^{it})^2}}} \\
    & \leq
    \frac{\abs{v_1(re^{it})} + o\parr{\abs{r-1}+\abs{t}}}
    {\sqrt{v_1(re^{it})^2 + c^2/4}} \\
    & \leq
    \frac{L}{\sqrt{L^2 + c^2/4}} + o\parr{\abs{r-1}+\abs{t}} < 1
\end{align*}
on a sufficiently small neighborhood of $(r,t)=(0,1)$. This establishes a uniform bound on the angle between the normal and tangential derivatives, $\frac{\partial \MAP}{\partial r}$ and $\frac{\partial \MAP}{\partial t}$. Lastly, continuity implies that $\norm{\frac{\partial \MAP}{\partial r}}$ and $\norm{\frac{\partial \MAP}{\partial t}}$ are bounded away from zero in a small enough neighborhood of $e^{i0}$. This ensures that $\MAP$ is a local homeomorphism around $e^{i0}$.
\end{proof}

\subsection{Proof of Theorem~\ref{thm:Choquet}} 
\label{sect:alt_proof_choquet}
\begin{proof} 
We begin with a simple observation: the solution to the Dirichlet problem
\begin{equation*} 
    \begin{cases}
        \Delta \MAP = 0 \qquad &\mbox{in}~\DISK\\
        \MAP = \chi_{x>0}(x,y) \qquad &\mbox{on}~\partial\DISK.
    \end{cases}
\end{equation*}
satisfies
$$ 
\MAP(x,0) = \frac{1}{2} + \frac{2}{\pi} \arctan{x}.
$$

Let $\BND$ be a simple curve enclosing a non-convex bounded domain $\IMG$ and let $\BND_a$ and $\BND_b$ denote two points whose connecting line segment is not contained in the interior of $\BND$. We define a homeomorphism $\BMAP:\partial \DISK \to \BND$ as follows: we pick two antipodal points $x,y\in \partial \DISK$ and consider a homeomorphism $\BMAP$ that maps $\BMAP(x)=\BND_a$, $\BMAP(y)=\BND_b$ and slows down in around these points in the sense illustrated in Figure~\ref{fig:choquet_construction}. 

\begin{figure}[h!]
  \centering
      \includegraphics[width=.9\columnwidth]{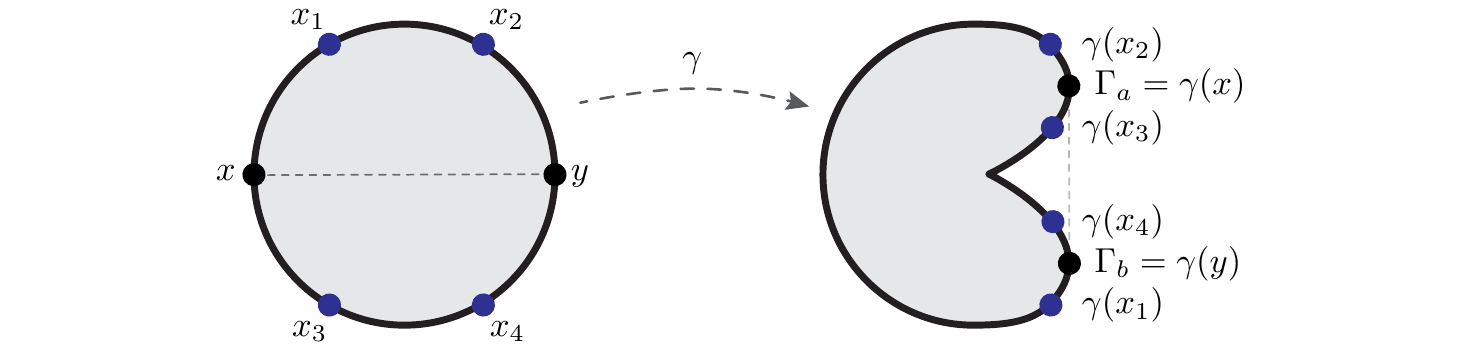}
  \caption{Implicit construction of $\BMAP$.}
  \label{fig:choquet_construction}
\end{figure}

If $x_1 \rightarrow x_2$ and $x_4 \rightarrow x_3$ in a way that preserves the symmetry of the construction, then the harmonic extension $\MAP$, solution of \eqref{eqn:cont_harmonic}, converges to 
$$ 
\MAP(z,0) \rightarrow  \parr{\frac{1}{2} + \frac{2}{\pi} \arctan{z}} \BND_a + \parr{\frac{1}{2} +-\frac{2}{\pi} \arctan{z}} \BND_b
$$ 
which converges to the straight line connecting $\BND_a$ and $\BND_b$. Thus, we obtain a homeomorphism $\BMAP$ (or rather an entire class) for which the harmonic extension $\MAP: \DISK \not\rightarrow \mbox{int}(\IMG)$.
\end{proof}

%%%%%%%%%%%%%%%%%%%%%%%%%%%%%%%%%%%%%%%%%%%%%%%%%%%%%%%%%%%%%%%%%%%%%%%%%

\section{Discrete-Harmonic Mappings: Proofs}

\subsection{An auxiliary lemma}
We begin with a simple geometric observation on intersection-free straight-line drawings of triangulations that will be used in the proofs.

Suppose $\parr{\y,\TT}$ is a proper intersection-free straight-line drawing  of a triangulation $\TT$. We define the cone spanned by the neighbors of a vertex $v_i$ by
$$
\brac{\NCone(\y,\TT)}(v_i) := \set{\sum_{j \in \NN_\TT(v_i)} \alpha_j \parr{y_j - y_i} \ :\ \alpha_j>0} \subseteq \Real^2.
$$

We note that in an intersection-free drawing of a triangulation, any \emph{interior} or \emph{strictly reflex} boundary vertex is strictly contained in the convex hull of its neighbors (a boundary vertex is a strictly reflex boundary vertex if its internal angle is strictly larger than $\pi$). Alternatively, this observation can be expressed in terms of the cones $\NCone(\y,\TT)$ as follows,

\begin{lemma} \label{lemma:SisR2NEW}
$\brac{\NCone(\y,\TT)}(v_i) = \Real^2$ for any interior or strictly reflex boundary vertex $v_i$ of the drawing $\parr{\y,\TT}$.
\end{lemma}

\subsection{Proof of Theorem~\ref{thm:nonConvexTutteGRPH}} 
\label{sect:proofNonConvexTutteGRPH}
Suppose $\y^*=\parr{y^*_1,...,y^*_n}\in \Real^{2 \times n}$ is a discrete-harmonic embedding into $\PP$ with positive weights $w_{ij}>0$, and assume that the cone condition $\brac{\Lap_{w}(\y^*,\TT)}(v_i)\in \CC_{\PP}(y^*_i)$ holds for all reflex boundary vertices. We wish to show that $\parr{\y^*,\TT}$ is an intersection-free straight-line drawing, triangulating the non-convex polygonal domain $\PP$. 

Our constructive proof is based on reduction to the convex case, where we apply Tutte's Theorem~\ref{thm:TutteNEW} to obtain an intersection-free triangulation of $\PP'=\conv(\PP)$, the convex hull of $\PP$. We break the proof into two main parts: (a) We extend the triangulation $\TT$, by adding new edges but no new vertices, into a triangulation $\TT'$ compatible with $\PP'$, the convex hull of $\PP$. We use the Jordan-Sch{\"o}nflies theorem to prove that the extended triangulation is valid and satisfies the assumptions required for applying Theorem~\ref{thm:TutteNEW} on the convex polygon $\PP'$. (b) We show how to construct positive weights $w_{ij}'>0$ on the edges of of the extended triangulation $\TT'$ such that the straight-line drawing $\parr{\y^*,\TT'}$ satisfies \eqref{eqn:discrete_harmonic} for all interior vertices of $\TT'$. Theorem~\ref{thm:TutteNEW} then implies that $\parr{\y^*,\TT'}$ is an intersection-free and thus so is $\parr{\y^*,\TT}$, which concludes the proof.

\subparagraph{(a) Convex extension.}
 We first extend the triangulation via convex completion. We consider the case where $\PP$ is non-convex, otherwise the assertion is readily satisfied by Theorem~\ref{thm:TutteNEW}.  Let $\PP'$ be the convex hull of $\PP$ and let $\PP^\Delta$ denote the polygon (or collection of polygons) enclosed between $\PP'$ and $\PP$. $\PP^\Delta$ is a union of simple polygons and can therefore be triangulated by adding new edges without adding new vertices, for example, by applying the ear clipping algorithm and the Two Ears Theorem \cite{meisters1975polygons}. Note that the vertices of $\PP^\Delta$ correspond to the boundary vertices of $\TT$, and are therefore a subset of its vertices. Therefore, such a triangulation of $\PP^\Delta$ is, in fact, an intersection-free straight-line drawing $\parr{\y^*,\TT^\Delta}$, with $\TT^\Delta = (V,E^\Delta,F^\Delta)$ encoding the triangles generated for triangulating $\PP^\Delta$. Figure~\ref{fig:convex_extension} illustrates the different components and notations used above. 

\begin{figure}[h!]
  \centering
  \includegraphics[width=.9\textwidth]{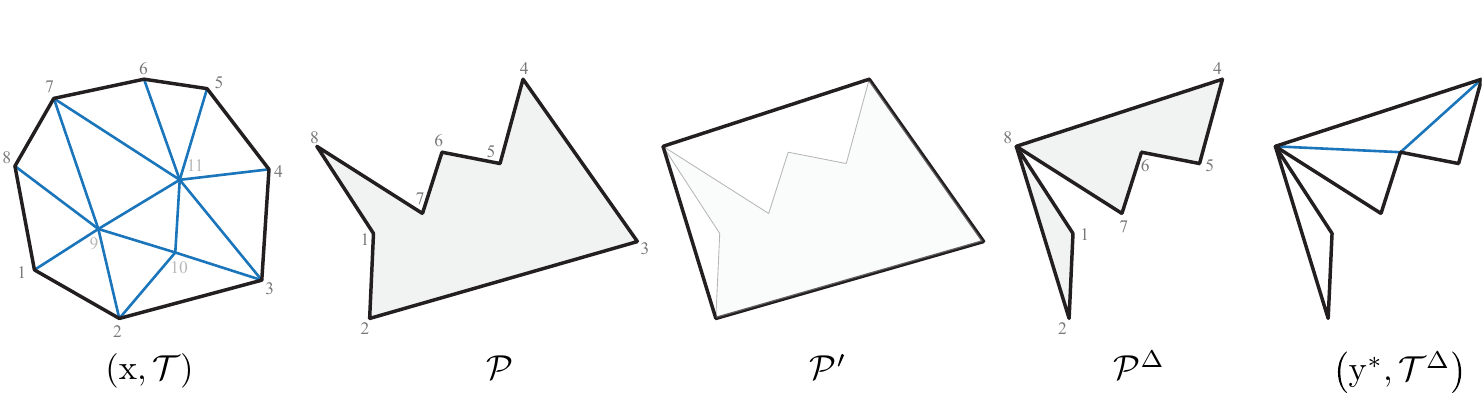}
  \caption{Extending the triangulation $\TT$ with respect to the convex hull $\PP'$ of $\PP$. }
  \label{fig:convex_extension}
\end{figure}

Note that, with these notations, $\TT$ and $\TT^\Delta$ share the same set of vertices, where in the latter we allow some vertices to remain unreferenced, \ie, some of the vertices of $\TT^\Delta$ do not belong to any face or edge. In turn, we define the extended triangulation $\TT'=(V,E',F')$ to be the union of $\TT$ and $\TT^\Delta$, obtained by combining their sets of faces $F'=F\cup F^\Delta$ and edges $E'=E\cup E^\Delta$.

To ensure we can use Theorem~\ref{thm:TutteNEW}  on the extended triangulation $\TT'$ we need to show that $\TT'$ is a 3-connected triangulation and that it has disk topology. 3-connectedness is obvious, as $\TT$ and $\TT'$ share the same set of vertices and, thought of as a graph, the degree of each vertex of $\TT'$ is at least that of $\TT$. To establish that $\TT'$ has disk topology we need to show that it admits a proper intersection-free drawing. To this end, we show that a drawing of newly introduced triangles $\TT^\Delta$ can be nicely ``stitched'' along the boundary of the intersection-free drawing $\parr{\x,\TT}$ in an intersection-free manner (though not necessarily using straight lines), as illustrated in Figure~\ref{fig:glue_extension}. 

To this end, we use the Jordan-Sch{\"o}nflies theorem \cite{thomassen1992jordan}:
\begin{theorem}[Jordan-Sch{\"o}nflies] \label{thm:jordan_schonflies} If $h$ is a homeomorphism of a simple closed curve $\gamma$ onto a simple closed curve $\gamma'$, then $h$ can be extended into a homeomorphism of the whole plane.
\end{theorem}

This theorem implies that there exists a homeomorphism of the entire plane $\psi:\Real^2\to\Real^2$ extending the correspondence between the boundary $\partial \PP$ of the polygon $\PP$ and the boundary $\partial \Omega_{\parr{\x,\TT}}$ of the drawing $\parr{\x,\TT}$. In turn, the intersection-free straight-line drawing $\parr{\y^*,\TT^\Delta}$ can be mapped via the homeomorphism $\psi$ into an intersection-free (not necessarily straight-line) drawing which we denote $\psi\parr{\y^*,\TT^\Delta}$. Note that the polygon $\Omega_{\parr{\x,\TT}}$ contains the drawing $\parr{\x,\TT}$ of $\TT$. On the other hand, the drawing $\parr{\y^*,\TT^\Delta}$ of $\TT^\Delta$ is entirely contained in the complement of $\PP$. Consequently, the Jordan-Sch{\"o}nflies theorem implies that the drawing $\psi\parr{\y^*,\TT^\Delta}$ of $\TT^\Delta$ and the drawing $\parr{\x,\TT}$ of $\TT$ do not overlap, except on $\partial \Omega_{\parr{\x,\TT}}$ where it ensures they are consistent. In turn, the triangulation $\TT'$, which is the union of the triangulations $\TT$ and $\TT^\Delta$, admits an intersection-free drawing by taking the union of the two drawings $\parr{\x,\TT}$ and $\psi\parr{\y^*,\TT^\Delta}$. See Figure~\ref{fig:glue_extension} for an illustration.

\begin{figure}[t!]
  \centering
  \includegraphics[width=.9\textwidth]{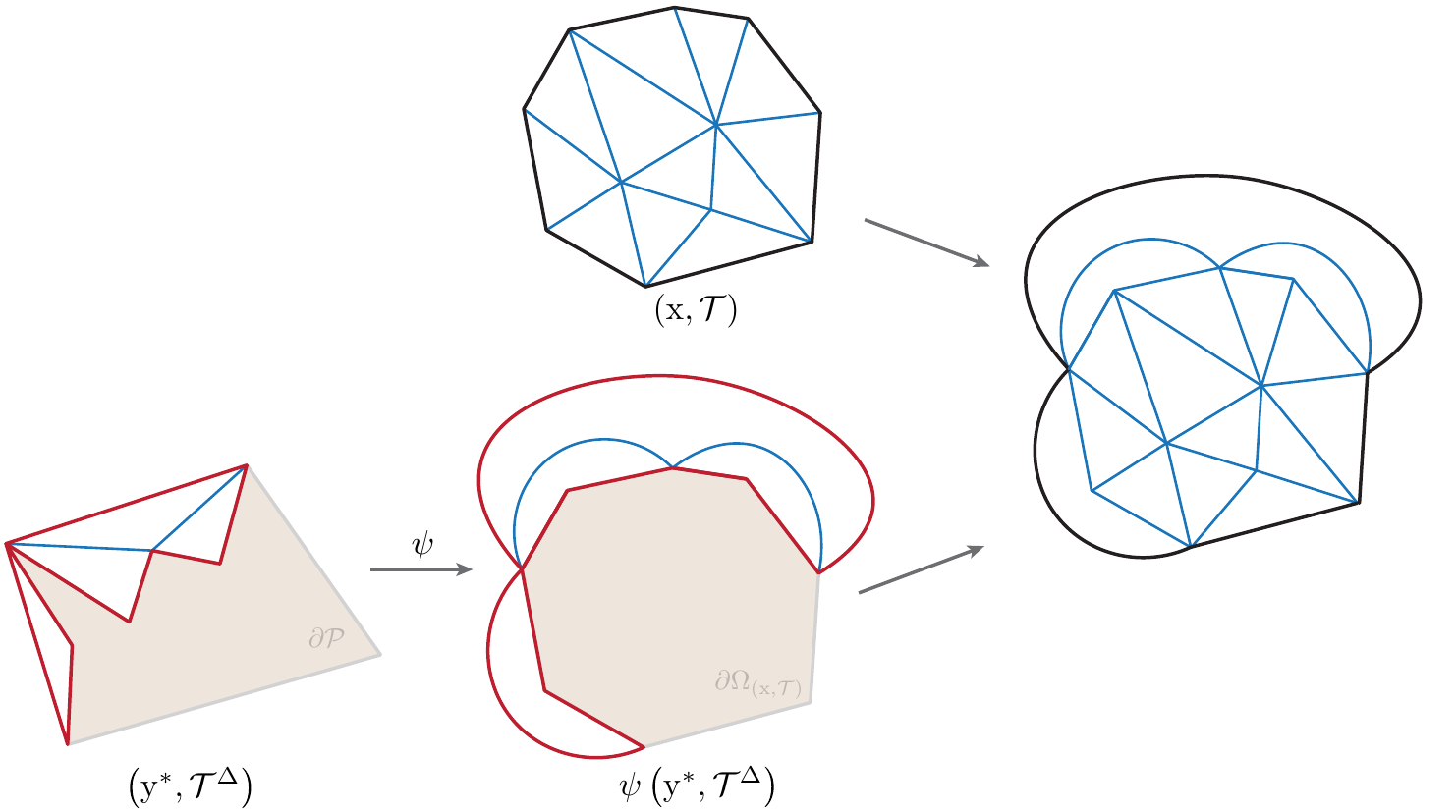}
  \caption{Using the Jordan-Sch{\"o}nflies theorem to ``stitch'' the newly introduced triangles $\TT^\Delta$ along the boundary of the intersection-free drawing $\parr{\x,\TT}$ in an intersection-free manner}
  \label{fig:glue_extension}
\end{figure}

\subparagraph{(b) Discrete-harmonic embedding of the extension.} We now focus on the extended triangulation $\TT'$ and the straight-line drawing $\parr{\y^*,\TT'}$ with vertex coordinates $\y^*\in \Real^{2 \times n}$. We will show that the straight-line drawing $\parr{\y^*,\TT'}$ is an intersection-free triangulation of the convex polygon $\PP'$. Consequently, since $\TT\subset\TT'$, this will show that $\parr{\y^*,\TT}$ is intersection-free as well, thus concluding the proof.

To show that $\parr{\y^*,\TT'}$ is intersection free, we construct positive weights $w_{ij}'>0$ on the edges of the extended triangulation $\TT'$ such that the vertex coordinates $\y^*$ satisfy \eqref{eqn:discrete_harmonic} for all interior vertices of $\TT'$. In turn, since $\y^*$ maps the boundary of $\TT'$ to the boundary of a convex polygonal domain $\PP'$, Theorem~\ref{thm:TutteNEW} shows that the straight-line drawing $\parr{\y^*,\TT'}$ is intersection-free.

To construct weights $w'=\parr{w_{ij}'}$ on the edges of the extended triangulation $\TT'$ we independently consider each vertex $v_i\in V$ and choose weights on its incident (directed) edges in $\TT'$ such that 
$$\brac{\Lap_{w}(\y^*,\TT')}(v_i) = \sum_{j \in \NN_{\TT'}(v_i)} w_{ij} \parr{y^*_j - y^*_i}=0.$$ 

\pagebreak
There are several cases to consider, see Figure~\ref{fig:proof_cases} for an illustration:

\begin{enumerate}
\item 
\emph{$v_i$ is a boundary vertex of the extension $\TT'$:} Its coordinate $y^*_i$ is fixed to a vertex of $\PP'$ regardless of the choice of weights $w_{ij}'$ and \eqref{eqn:discrete_harmonic} need not be satisfied. 

\item 
\emph{$v_i$ is interior to the original triangulation $\TT$:} In this case we copy the original weights and set $w'_{ij}=w_{ij}>0$ for all $j\in \NN_\TT(v_i)$. Since $v_i$ has the same neighborhood in both $\TT$ and its extension $\TT'$, we have $\brac{\Lap_{w}(\y^{*},\TT')}(v_i)=0$.

\item
\emph{$v_i$ is interior to $\TT'$ but belongs to the boundary of $\TT$:} As such, $v_i$ is mapped in the drawing $\parr{\y^*,\TT}$ to either a strictly convex or a reflex vertex of the polygon $\PP$:
\begin{enumerate}
\item 
If $v_i$ is mapped to a strictly convex vertex of $\PP$ (whose internal angle is strictly smaller than $\pi$), it is a strictly reflex vertex of the complementary polygon $\PP^\Delta$ with respect to the convex hull. Applying Lemma~\ref{lemma:SisR2NEW} to the triangulation $\parr{\y^*,\TT^\Delta}$ of $\PP^\Delta$ implies that $\brac{\NCone(\y^*,\TT^\Delta)}(v_i) = \Real^2$ with respect to the neighbors of $v_i$ in $\TT^\Delta$. In turn, since $\NN_{\TT^\Delta}(v_i)\subseteq\NN_{\TT'}(v_i)$ we have that $\brac{\NCone(\y^*,\TT')}(v_i) = \Real^2$ in the extended triangulation $\TT'$ and hence there exist positive weights $w_{ij}'>0$ such that $\brac{\Lap_{w}(\y^{*},\TT')}(v_i)=0$.

\item
Lastly, consider the case where $v_i$ is mapped to a reflex vertex of $\PP$. In this case, by the assumptions made in the theorem, we know that $v_i$ satisfies the cone condition $\brac{\Lap_{w}(\y^*,\TT)}(v_i)\in \CC_{\PP}(y^*_i)$. Since $v_i$ is a reflex vertex of $\PP$, it is a convex vertex with respect to the triangulation $\parr{\y^*,\TT^\Delta}$ of the complementary polygon $\PP^\Delta$. In this case, we note that $\brac{\NCone(\y^*,\TT^\Delta)}(v_i)$ and $\CC_{\PP}(y^*_i)$ are opposite cones, see Figure~\ref{fig:proof_cases}. Therefore $-\brac{\Lap_{w}(\y^*,\TT)}(v_i) \in \brac{\NCone(\y^*,\TT^\Delta)}(v_i)$; namely, there exist positive weights $w_{ij}^\Delta>0$ on the edges of $\TT^\Delta$ such that
$$
\sum_{j \in \NN_{\TT^\Delta}(v_i)} w_{ij}^\Delta \parr{y^*_j - y^*_i} = -\brac{\Lap_{w}(\y^*,\TT)}(v_i) = - \sum_{j \in \NN_\TT(v_i)} w_{ij} \parr{y^*_j - y^*_i},
$$
where the second equality is simply the definition of $\Lap_{w}(\y^*,\TT)$. We then choose $w_{ij}'$ to be either $w_{ij}$, $w_{ij}^\Delta$ or their sum (for the boundary edges of $\TT$ incident to $v_i$), which satisfy $\brac{\Lap_{w}(\y^{*},\TT')}(v_i)=0$. \end{enumerate}
\end{enumerate}

\begin{figure}[t!]
  \centering
  \includegraphics[width=.9\textwidth]{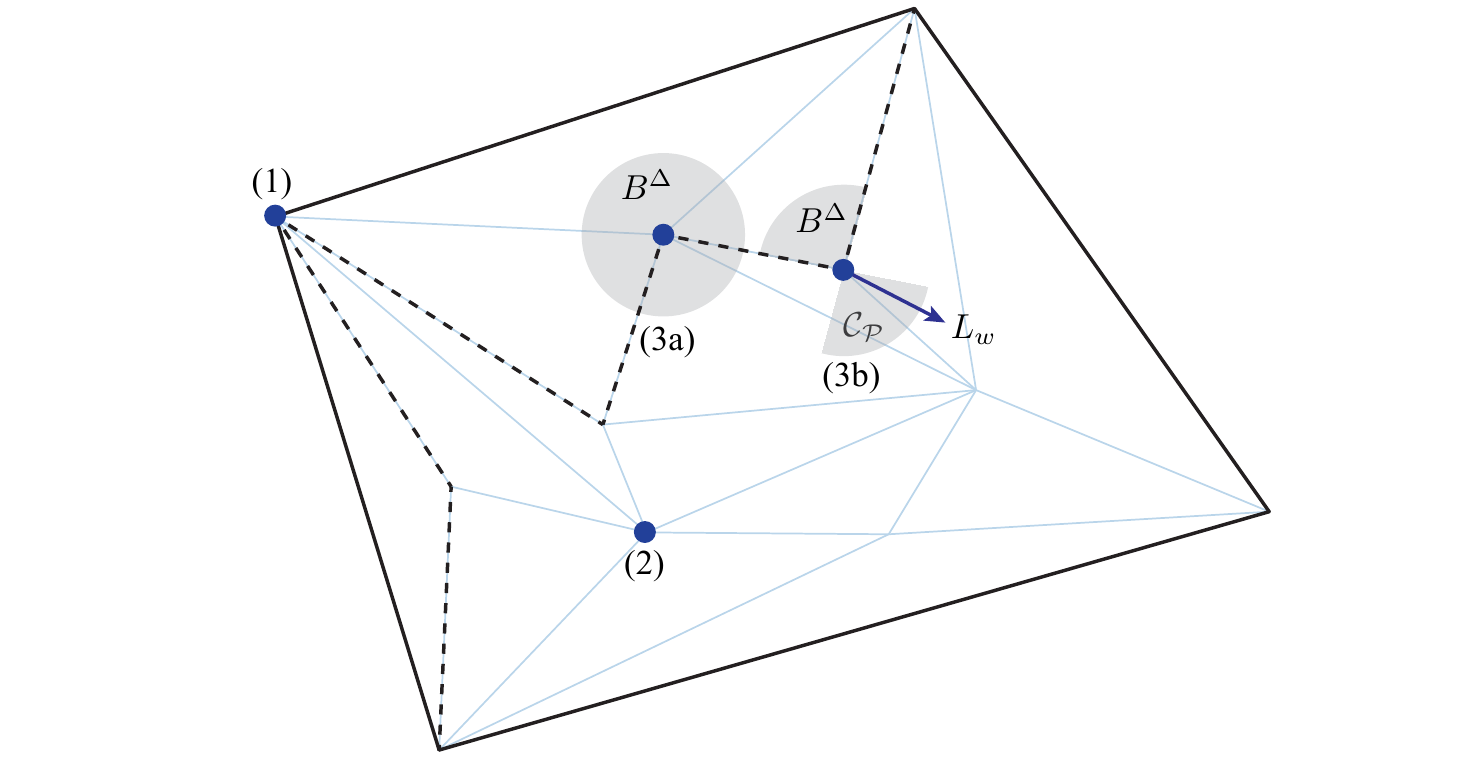}
  \caption{The different cases considered in the proof of Theorem~\ref{thm:nonConvexTutteGRPH} and key objects discussed in each case.}
  \label{fig:proof_cases}
\end{figure}

With weights chosen as described above, $\brac{\Lap_{w}(\y^{*},\TT')}(v_i)=0$ for all interior vertices $v_i$ of $\TT'$.  Theorem~\ref{thm:TutteNEW} then implies that the straight-line drawing $\parr{\y^*,\TT'}$ is intersection-free and thus so is $\parr{\y^*,\TT}$, which concludes the proof.

\subsection{Proof of Theorem~\ref{thm:nonConvexTuttePWL}}
\label{sect:proofNonConvexTuttePWL}

Suppose $\phi_{\parr{\x,\TT}\to\parr{\y,\TT}}$ is a piecewise-linear homeomorphism of $\Omega_{\parr{\x,\TT}}$ onto $\PP$. We want to show the existence of positive weights $w_{ij}>0$ such that $\y$ satisfies \eqref{eqn:discrete_harmonic} and the cone condition $\brac{\Lap_{w}(\y,\TT)}(v_i)\in \CC_{\PP}(y_i)$ holds for all reflex boundary vertices $v_i$ of $\TT$. Since $\phi_{\parr{\x,\TT}\to\parr{\y,\TT}}$ is a homeomorphism, the straight-line drawing $\parr{\y,\TT}$ is intersection-free, being the image of the intersection-free drawing $\parr{\x,\TT}$. Lemma~\ref{lemma:SisR2NEW} then implies that $\brac{\NCone(\y,\TT)}(v_i) = \Real^2$ for all interior or strictly reflex boundary vertices. Thus, for each interior vertex $v_i$, we have $0\in \brac{\NCone(\y,\TT)}(v_i)$ that implies the existence of positive weights $w_{ij}>0$ that satisfy \eqref{eqn:discrete_harmonic}. Similarly, there exist positive weights that satisfy the cone condition for all strictly reflex boundary vertices. Lastly, for boundary vertices that are both reflex and convex the associated boundary cone is a open half-space, and thus the cone condition is satisfied for any positive weights.

For the other direction, we use a result by Whitney \cite{whitney1932congruent} regarding the uniqueness of embedding of 3-connected graphs. Whitney's result implies that once the outer face is chosen, a 3-connected graph has a unique embedding homeomorphism of the plane \cite{dolev1983planar}.

Suppose $\y$ is discrete-harmonic into $\PP$ satisfying \eqref{eqn:discrete_harmonic} and the cone condition $\brac{\Lap_{w}(\y,\TT)}(v_i)\in \CC_{\PP}(y_i)$ holds for all reflex boundary vertices $v_i$ of $\TT$. Consequently, Theorem~\ref{thm:nonConvexTutteGRPH} asserts that the straight-line drawing $\parr{\y,\TT}$ is intersection-free. Note that $\parr{\x,\TT}$ and $\parr{\y,\TT}$ are both proper drawings of the triangulation $\TT$, that is, their external face corresponds to the boundary edges of $\TT$. Therefore, Whitney's result implies that the piecewise-linear map $\phi_{\parr{\x,\TT}\to\parr{\y,\TT}}$ that maps the drawing $\parr{\x,\TT}$ onto $\parr{\y,\TT}$ is a homeomorphism.

\subsection{Proof of Theorem~\ref{thm:dis_AN_PWL}}
\label{sect:proofdis_AN_PWL}

If $\phi_{\parr{\x,\TT}\to\parr{\y,\TT}}:\Omega_{\parr{\x,\TT}} \to \Real^2$ is a homeomorphism then clearly $\det D\phi_{\parr{\x,\TT}\to\parr{\y,\TT}}>0$ on the interior of all triangles. 

For the converse, suppose that $\y=\parr{y_1,...,y_n}\in \Real^{2 \times n}$ is a discrete-harmonic embedding into a simple polygonal domain $\PP$ satisfying \eqref{eqn:discrete_harmonic} with positive weights $w_{ij}>0$ and that $\det D\phi_{\parr{\x,\TT}\to\parr{\y,\TT}}>0$ on the interior of all boundary triangles. We will use Theorem~\ref{thm:nonConvexTuttePWL} to conclude that $\phi_{\parr{\x,\TT}\to\parr{\y,\TT}}$ is a homeomorphism. 

The proof is immediate if reflex boundary vertices satisfy the cone condition; if the cone condition $\brac{\Lap_{w}(\y,\TT)}(v_i)\in \CC_{\PP}(y_i)$ holds for all reflex boundary vertices $v_i$ of $\TT$ then Theorem~\ref{thm:nonConvexTuttePWL} implies that $\phi_{\parr{\x,\TT}\to\parr{\y,\TT}}$ is a homeomorphism of $\Omega_{\parr{\x,\TT}}$ onto $\PP$. Otherwise, we separately address each reflex boundary vertex that does not satisfy the cone condition: we show that the weights of its incident edges may be modified, with no change to the map itself, so that the cone condition is satisfied; then, Theorem~\ref{thm:nonConvexTuttePWL} could be used to conclude the proof.

Towards this end, we have the following lemma:
\begin{lemma} \label{lemma:boudnary_positive_det}
Let $v_i$ be a reflex boundary vertex, that is, a boundary vertex of $\TT$ mapped to a reflex vertex of $\PP$. If $\det D\phi_{\parr{\x,\TT}\to\parr{\y,\TT}}>0$ on the interior of all triangles incident to $v_i$, then the cone  
$$
\brac{\NCone(\y,\TT)}(v_i) := \set{\sum_{j \in \NN_\TT(v_i)} \alpha_j \parr{y_j - y_i} \ :\ \alpha_j>0}
$$
satisfies $\brac{\NCone(\y,\TT)}(v_i) = \Real^2$.
\end{lemma}
\begin{proof}
Suppose, without loss of generality, that $v_1,\ldots,v_K$ are the vertices adjacent to $v_i$, ordered in clockwise order with respect to the intersection-free straight-line drawing $\parr{\x,\TT}$. Denote by $Y_j = y_j - y_i$ the image of the edges incident to $v_i$ and let $\alpha_j$ be the angle from $Y_j$ to $Y_{j+1}$ in the clockwise direction. Let $\beta$ denote the clockwise angle between $Y_1$ and $Y_K$, which correspond to the boundary edges incident to $v_i$, and note that since $v_i$ is a reflex vertex $\pi<\beta<2\pi$. These notations are illustrated in Figure~\ref{fig:det_proof_notations}.

\begin{figure}[t!]
  \centering
  \includegraphics[width=.9\textwidth]{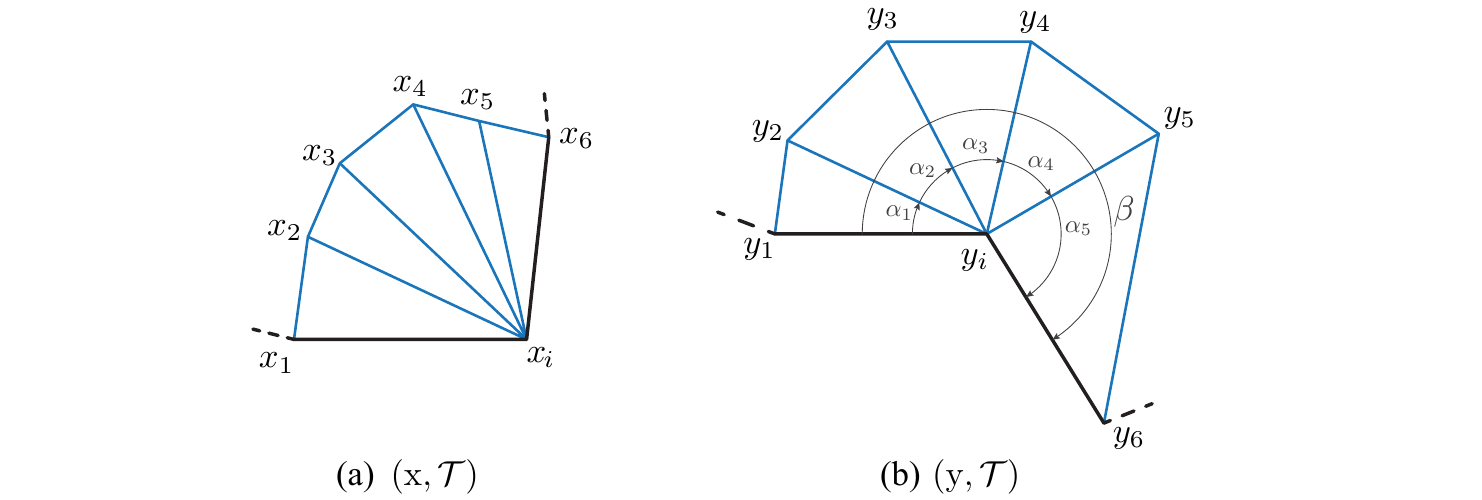}
  \caption{The notations used in the proof of Lemma~\ref{lemma:boudnary_positive_det}. (a) shows the neighborhood of the boundary vertex $v_i$ as realized in the intersection-free straight-line drawing $\parr{\x,\TT}$. (b) shows in image of $\phi_{\parr{\x,\TT}\to\parr{\y,\TT}}$ wherein $v_i$ is realized as a reflex boundary vertex.}
  \label{fig:det_proof_notations}
\end{figure}

Since $\phi_{\parr{\x,\TT}\to\parr{\y,\TT}}$ is continuous and orientation preserving on the $K-1$ triangles incident to $v_i$, we have that $0<\alpha_1,\ldots,\alpha_{K-1}<\pi$ and $\sum_{j=1}^{K-1} \alpha_j = \beta + 2\pi n$, $n\in \Natural$. Figure~\ref{fig:det_proof_angle_condition} visualizes these conditions.

\begin{figure}[t!]
  \centering
  \includegraphics[width=.9\textwidth]{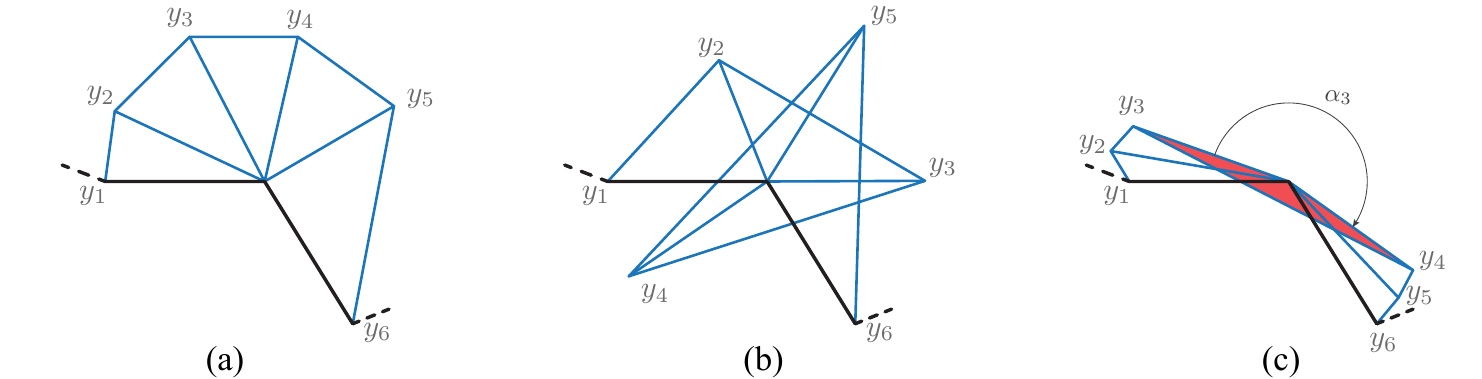}
  \caption{The angle conditions, $0<\alpha_j<\pi$ and $\sum \alpha_j > \pi$, are satisfied in (a) and (b). (b) depicts winding around the boundary vertex, where $\sum \alpha_j = \beta + 2\pi$. (c) shows an infeasible configuration: angles between consecutive edges cannot exceed $\pi$, otherwise a boundary triangle (shown in red) reverses its orientation, in contradiction to the assumption that $\det D\phi_{\parr{\x,\TT}\to\parr{\y,\TT}}>0$ on boundary triangles.
  }
  \label{fig:det_proof_angle_condition}
\end{figure}

By way of contradiction suppose that $\brac{\NCone(\y,\TT)}(v_i) \neq \Real^2$, that is, there exists a non-trivial $z\in \Real^2$ such that $z\notin \textrm{cl}\parr{\brac{\NCone(\y,\TT)}(v_i)}$. Namely, for such $z$ the linear system $\sum \alpha_j Y_j = z$ cannot be satisfied with $\alpha_j\geq0$. In turn, Farkas' Lemma \cite{rockafellar1970convex} implies that there exists $c\in\Real^2$ such that $c^T z < 0$ and $c^T Y_j \geq 0$ for all $j=1,\ldots,K$. This, however, is impossible as the angles between $Y_j$ satisfy $0<\alpha_j<\pi$ and $\sum_{j=1}^{K-1} \alpha_j > \pi$, which means that when $j$ goes from $1$ to $K$, $c^TY_j$ must change sign for some $j$.
\end{proof}

Now, suppose that $v_k$ is a reflex boundary vertex were the cone condition does not hold. Lemma~\ref{lemma:boudnary_positive_det} implies that $\brac{\NCone(\y,\TT)}(v_k) = \Real^2$. Consequently, we have the freedom to choose weights $w_{kj}>0$ for the directed edges incident to $v_k$ such that 
$$\sum_{v_j \in \NN(v_k)} w_{kj} \parr{y_j - y_k} = \brac{\Lap_{w}(\y,\TT)}(v_k)\in \CC_{\PP}(y_k).$$ 
Note that we are free to do this since $v_k$ is fixed on the boundary and the weights are not required to be symmetric; consequently, such reassignment of $w_{kj}>0$ does not change the map $\phi_{\parr{\x,\TT}\to\parr{\y,\TT}}$. Repeating the same argument for all reflex boundary vertices ensures that the conditions of Theorem~\ref{thm:nonConvexTuttePWL} are satisfied and, consequently, that $\phi_{\parr{\x,\TT}\to\parr{\y,\TT}}$ is a homeomorphism of $\Omega_{\parr{\x,\TT}}$ onto $\PP$.

%%%%%%%%%%%%%%%%%%%%%%%%%%%%%%%%%%%%%%%%%%%%%%%%%%%%%%%%%%%%%%%%%%%%%%%%%

\end{document}